\documentclass[12pt]{amsart}
\usepackage{setspace}
\usepackage{amsfonts}
\usepackage{amsmath}
\usepackage{amssymb}
\usepackage{amsthm}
\usepackage{mathrsfs}
\usepackage[all]{xy}

\numberwithin{equation}{section}

\newcommand{\proofend}{\hfill \hbox{\vrule width 5pt height 5pt depth
0pt}}
\newcommand{\R}{\mathbb{R}}
\newcommand{\C}{\mathbb{C}}
\newcommand{\Z}{\mathbb{Z}}
\newcommand{\N}{\mathbb{N}}
\newcommand{\Q}{\mathbb{Q}}

\newcommand{\proj}{\mathbb{P}}

\newtheorem{thm}{Theorem}
\newtheorem{lemma}{Lemma}[section]
\newtheorem{corol}[thm]{Corollary}

\newtheorem{propo}[lemma]{Proposition}

\begin{document}

\title[Convergence and periodic points]{Convergence to the Mahler measure and the distribution of periodic points for  algebraic Noetherian $\Z^d$-actions}
\author{Vesselin Dimitrov}
\address{Yale University Math. Dept. \\ 10 Hillhouse Avenue \\ CT~06520--8283 }
\email{vesselin.dimitrov@yale.edu}

\begin{abstract}
We prove a sub-Liouville bound, up to a uniformly bounded exceptional set, on the distance from an $N$-torsion point of $\mathbb{G}_{m/\Q}^d$ to an algebraic subset, under a fixed Archimedean place $\bar{\Q} \hookrightarrow \C$. As a consequence, for all non-zero integer Laurent polynomials $P$ in $d$ commuting variables, we prove that the averages of $\log{|P(\mathbf{x})|}$ over $\boldsymbol{\zeta} \in \mu_N^d$, $P(\boldsymbol{\zeta}) \neq 0$, converge as $N \to \infty$ to the Mahler measure of~$P$.

By the work of  B. Kitchens, D. Lind, K. Schmidt and T. Ward, this convergence consequence amounts to the following statement in dynamics: For every Noetherian $\Z^d$-action $T : \Z^d \to \mathrm{Aut}(X)$ by automorphisms of a compact abelian group $X$ having a finite topological entropy $h(T)$, the annihilator $\mathrm{Per}_N(T)$ of $N \cdot \Z^d$ has $e^{(1+o(1))h(T)N^d}$ connected components,  as $N \to \infty$. Moreover, it follows that all weak-$*$ limit measures of the push-forwards of the Haar measures on $\mathrm{Per}_N(T)$, under any a sequence of positive integers $N$, are  measures of maximum entropy $h(T)$.

Combined with  work of Thang Le, this solves the abelian case of the problem of  the asymptotic growth of torsion in the homology of congruence covers  of a fixed finite simplicial complex.
 We also give an indication on how an alternative, completely different route to the convergence result and its consequences is also possible through Habegger's recent work on Diophantine approximation to definable sets.
 \end{abstract}

\maketitle

\newpage

\section{Introduction} \label{intro}

 Ergodic and dynamical systems theory have been marked over the past three decades by an extensive interaction with   central problems of classical Diophantine analysis. This continues a much older  tradition begun by Hermann Weyl in a 1914 paper~\cite{weyl} that he  conceived as ``an application of number theory'' to Boltzmann's ergodic hypothesis ({\it Sur une application de la th\'eorie des nombres \`a la m\'ecanique statistique et la th\'eorie des perturbations}). In more recent times, which have witnessed ergodic methods penetrating deeply into all aspects of number theory, the direction of implications is usually reversed, and another line of inquiry, emphasizing higher rank and rigidity phenomena with their relation to Diophantine approximations, was started with Furstenberg's influential 1966 paper~\cite{furstenberg}, where the famous $\times 2, \times 3$ problem was formulated. Our present paper owes something to both  traditions, by using number theory to verify a hypothesis on the distribution of periodic trajectories  for a class of higher rank dynamical systems that contains Furstenberg's $\times 2, \times 3$ system as a very special case: the Noetherian $\Z^d$-actions by automorphisms of a compact abelian group. In the rank one case this problem was solved long ago, by Lind in~\cite{lindqh}. Our result addresses the higher rank situation, where the number theory is less understood, and could, conceivably,  turn out to behave differently.

 A superb introduction to ergodic ideas in number theory is in the book~\cite{einsidlerward} by Einsiedler and Ward. The reader will find in our final section~\ref{appo} an overview of
  Diophantine/dynamical pairs of mathematically equivalent problems, placing our result in a broad yet, we hope, sufficiently focused context.

In this paper, we treat  a pair of equivalent questions about \emph{algebraic $\Z^d$-actions}: the $\Z^d$-actions $T : \Z^d \to \mathrm{Aut}(X)$ by automorphisms of a compact  group $X$. This addresses the problem left open in T. Ward's thesis~\cite{ward} and then in K. Schmidt's book~\cite{schmidt}, and the papers~\cite{previous,atoral} of D. Lind, K. Schmidt and E. Verbitskiy:  when $X$ is abelian, and under an appropriate --- and necessary --- Noetherianness assumption (the descending chain condition on closed invariant subgroups), \emph{does the exponential growth rate of the size of the component group of the subgroup $\mathrm{Per}_{\Gamma}(T) \subset X$ of $\Gamma$-periodic points, indexed by an order sublattice $\Gamma \subset \Z^d$, converge to the topological entropy of the system, assuming the latter is finite?} We answer this in the affirmative in the symmetric case $\Gamma = N \cdot \Z^d$ of the cubical sublattices. On  applying another paper~\cite{wardexp} of Ward, this furthermore proves  that, when $(X,T)$ has a completely positive entropy,
  the subgroups $\mathrm{Per}_N(T) :=  \mathrm{Per}_{N \cdot \Z^d}(T)\subset X$ become equidistributed in the Haar measure of the group $X$, as $N \to \infty$. {\it Completely positive (finite) entropy} is the natural condition as it means equivalently that the Haar measure is the unique probability measure of maximum entropy: this is Berg's theorem, proved in this situation by Lind, Schmidt and Ward~\cite{lindschmidtward}. In the general case, when $(X,T)$ supports more than one maximum entropy measure, the conclusion is that all  weak-$*$ limits over some such sequence of levels $N$ are measures of maximum entropy.

   As  we indicate in sections~\ref{statement} and~\ref{hlawka} below,  an entirely different solution, coming again from the number theory side, is implicit in Habegger's very recent manuscript~\cite{habegger}. Habegger does not note the dynamical connection, but the corresponding convergence problem follows easily from his general estimate on the number of rational points lying very near to a subset of $\R^d$ definable in a polynomially bounded $o$-minimal expansion of the real numbers, in conjunction with the Ax and Koksma-Hlawka theorems.   Habegger's and our Diophantine results are quite different, and generalize in completely disjoint directions.
    Neither of the two appears capable of addressing the growth rate and equidistribution of the general groups $\mathrm{Per}_{\Gamma}(T) \subset \mathrm{Per}_{\exp(\Gamma)}(T)$, and even less the equidistribution of individual periodic trajectories in $\mathrm{Per}_N(T)$ that are ``long enough'' with respect to a certain other characteristic of the orbit, that we shall not discuss in this paper, which following Einsiedler, Lindenstrauss, Michel and Venkatesh~\cite{individual} could be called the ``discriminant'' (or ``denominator'') of the orbit; see also chapter~21 of McMullen~\cite{ctm} for a discussion of the basic case of the doubling map of the circle. On the other hand, as the groups $\mathrm{Per}_N(T)$ considered exhaust the periodic points of the $\Z^d$-action, our result can be seen as an averaged form of the growth and equidistribution of periodic trajectories of an algebraic $\Z^d$-action.

 In the case that presently concerns us, the substance of the problem lies in the systems that exhibit only a partially hyperbolic behavior; the case of expansive $\Z^d$-actions has no Diophantine content, and was settled already in Ward~\cite{ward,wardexp}.
The Diophantine problem was recognized by Lind in~\cite{lindqh} to emerge from the dual description (or standard model) of an algebraic $\Z^d$-action, by means of commutative algebra. A $\Z^d$-action $T : \Z^d \to \mathrm{Aut}(X)$ by automorphisms of a compact abelian group $X$ is given by a module $\mathfrak{M} := \widehat{X}$ over the Laurent series ring $R_d := \Z[\Z^d] \cong \Z[x_1^{\pm 1}, \ldots, x_d^{\pm 1}]$. The  descending chain condition mentioned amounts precisely to the Noetherianness condition of the $R_d$-module $\mathfrak{M}$. When this is the case (but not without the condition, nor when the group $X$ is non-abelian), Kitchens and Schmidt (see \cite{kitchensschmidt}, Cor.~4.8) proved that the periodic points of $T$ are dense. One then expects the situation to be similar to the basic case of a toral automorphism, where the dynamics is given by an invertible integer matrix $A \in \mathrm{GL}(m,\Z)$; its entropy $h$ equals the Mahler measure $m(P) = \int_{S^1} \log{|P(z)|} \, d\theta = \sum_{i=1}^m \log^+{|\alpha_i|}$ of the companion polynomial $P(T) = \det( T \cdot I_m - A) = \prod_{i=1}^m (T - \alpha_i)$; the number of points with finite period $N$ equals $P_N := |\det(A^N - I_m)| = \prod_{i=1}^m |\alpha_i^N - 1|$; and the convergence $\lim_{N \to \infty} \frac{1}{N} \log{P_N} = h$ amounts precisely to A. O. Gelfond's theorem that an algebraic integer $\alpha$ of unit modulus $|\alpha| = 1$ may not be exponentially approached by an $N$-th root of unity $\zeta_N \neq \alpha$: $-\log{|\alpha - \zeta_N|} = o(N)$, as $N \to \infty$. The extension of this to solenoids ($A \in \mathrm{GL}(m,\Q)$) covers the case of an arbitrary algebraic $\alpha$ of unit modulus. For a lucid and leisurely  treatment of the $d=1$ case, including Yuzvinskii's calculation of the entropy, we refer the reader to Everest and Ward's book~\cite{everestward}, in addition to Lind's original paper~\cite{lindqh}.

That the general situation is indeed similar, only much more complicated,  emerged from the work of Kitchens and Schmidt~\cite{kitchensschmidt}, Ward~\cite{ward} and Lind, Schmidt and Ward~\cite{lindschmidtward}. By an addition formula for the entropy, due to Yuzvinskii and Thomas in the rank one case, a \emph{devissage} reduces the problem to the case that $\mathfrak{M} = R_d/I$, where $I$ is an ideal of the Laurent ring: the case of a cyclic action. Then the  entropy of $T$ turns out to be zero unless the ideal $I = (P)$ is principal, in which case it equals the Mahler measure $h(T) = m(P) := \int_{(S^1)^d} \log{|P(\mathbf{z})|} \, d\boldsymbol{\theta}$ (if $P \neq 0$, and infinity if $P = 0$). In such a way, Lind, Schmidt and Ward (see also the introduction in Lind, Schmidt and Verbitskiy~\cite{atoral}) obtain the equivalence of the following  dynamical and  Diophantine statements. With $\Gamma$ ranging over all finite index subgroups of $\Z^d$, let $\langle \Gamma \rangle := \min\{ \| \mathbf{m}\| \mid  \mathbf{0} \neq \mathbf{m} \in \Gamma \}, \, \|\mathbf{m}\| := \max(|m_1|,\ldots,|m_d|)$, and denote by $P_{\Gamma}(T)$ the number of connected components of the group of $\Gamma$-periodic points for $T$. Let $h(T)$ be the topological entropy of the system $(X,T)$, which coincides with the metric entropy for the Haar measure. On the dynamical side we expect:

\begin{quote}
(A) \hspace{1cm} {\it Assume $T$ has finite topological entropy and satisfies the d.c.c.: every descending chain of closed invariant subgroups of $X$ is stationary. Then,
$$
\lim_{\Gamma: \, \langle \Gamma \rangle \to \infty} \frac{1}{[\Z^d:\Gamma]} \log{P_{\Gamma}(T)} = h(T).
$$}
\end{quote}

 In general, as explained in~\cite{lindschmidtward} on page~619, the Noetherian and finite entropy conditions are certainly necessary in such a statement; nothing could be said for non-Noetherian systems, where the growth rate need not converge and any rate between $0$ and $h(T)$ may occur (see, however, Baier, Jaidee, Stevens and Ward~\cite{exotic}). Under the Noetherian and finite entropy conditions, Schmidt proves in section~21 of~\cite{schmidt} a weaker statement identifying $h(T)$ with the limit \emph{supremum} over sublattices $\Gamma \subset \Z^d, \langle \Gamma \rangle \to \infty$. He first proves the upper bound in a direct application of the topological definition of $h(T)$, and then uses Gelfond's theorem and the formula for the entropy to exhibit a particular sequence $\{ \Gamma_i \}$  along which $\frac{1}{[\Z^d:\Gamma_i]} \log{|P_{\Gamma_i}|} \to h$. This sequence is very special; in particular, the finite subgroup $\mathbb{G}_m^d[\Gamma_i]$ is of the form $\mu_{a_1} \times \cdots \times \mu_{a_d}$, where $a_d \to \infty$ with respect to $a_1,\ldots,a_{d-1}$; see section~\ref{baker} for more on this.

The general problem (A) thus amounts to establishing a \emph{lower} bound on the growth of periodic points, and as already mentioned, the same applies to equidistribution. Similarly to the Brauer-Siegel theorem and to questions on orbit growth in the arithmetic dynamics of rational maps (for the latter, see Silverman~\cite{domrat,silvab}; Kawaguchi-Silverman~\cite{kawsil1,kawsil2}; Dimitrov~\cite{vesselin,vesselinadd}), the upper bounds come easily, and the question of the lower bound proves to be a subtle problem related intrinsically to Diophantine Approximations.

In the case at hand, the cited work of Kitchens, Schmidt, Lind and Ward renders (A) precisely equivalent to the following multidimensional extension of Gelfond's result. Note that as $\langle \Gamma \rangle \to
\infty$, the finite group $\mathbb{G}_m^d[\Gamma] := \{\boldsymbol{\zeta} \mid \boldsymbol{\zeta}^{\Gamma} = 1\}$ is equidistributed in the Haar measure $d\boldsymbol{\theta}$ of the torus $(S^1)^d$.

\begin{quote}
(B) \hspace{1cm} {\it Let $P \in \Z[x_1^{\pm 1}, \ldots, x_d^{\pm 1}] \setminus \{0\}$ be a non-zero integer Laurent polynomial. Then
$$
\lim_{\Gamma: \, \langle \Gamma \rangle \to \infty} \frac{1}{|\mathbb{G}_m^d[\Gamma]|} \sum_{\substack{\boldsymbol{\zeta} \in \mathbb{G}_m^d[\Gamma] \\ P(\boldsymbol{\zeta}) \neq 0}}  \log{|P(\boldsymbol{\zeta})|} = \int_{(S^1)^d} \log{|P(\mathbf{z})|} \, d\boldsymbol{\theta}.
$$
}
\end{quote}

This is obvious in the case that the hypersurface $\{P = 0\} \subset \mathbb{G}_m^d(\C)$ has empty intersection with the real torus $(S^1)^d$. From the dynamical point of view, this happens precisely when the cyclic $\Z^d$-action dual to the $R_d$-module $R_d / (P)$ is expansive.
By the alluded {\it devissage} procedure, Lind, Schmidt and
 Ward were thus able to prove (A) for all expansive systems $(X,T)$. We note that, for expansive systems, both  the finiteness of the entropy and the Noetherianness conditions are automatic: see Kitchens and Schmidt~\cite{kitchensschmidt}, Th.~5.2.

Further progress was made by Lind, Schmidt and Verbitskiy in~\cite{previous}, who gave two proofs of (B) in the case that the zero locus $\{P = 0\}$ intersects the torus $(S^1)^d$ in a finite set.  The first  is Diophantine, and consists of the observation that since the finitely many intersection points are necessarily algebraic, an application of Gelfond's (one-dimensional) theorem suffices to cover that case. More significant is their second proof, purely dynamical, by means of a construction of rapidly decaying homoclinic points.
In a subsequent paper~\cite{atoral}, the  same authors then extended their dynamical method to prove (B) in the ``generic'' case that the intersection locus $\{P = 0\} \cap (S^1)^d$ has real codimension at least two in the torus. The dynamical equivalent of this ``atoral'' hypothesis turns out to be precisely the existence of summable homoclinic points, and that is what allowed the authors to bypass the delicate Diophantine issues about torsion points getting extremely close to the zero locus $\{P = 0\}$.

In the present paper, we solve (B) for all Laurent polynomials $P$ in the case that $\Gamma$ runs over the cubical sublattices $N \cdot \Z^d$, so that the average is taken over all $d$-tuples of $N$-th roots of unity. It follows from the work of Kitchens, Lind, Schmidt and Ward cited above that (A) is true when all the $N \cdot \Z^d$-periodic points are taken together.

\begin{thm} \label{convergence}
  The hypotheses $(A)$ and $(B)$ are true for $\Gamma = N \cdot \Z^d$, $N \to \infty$.
\end{thm}

The precise statement of the equidistribution consequence about periodic trajectories is as follows. The proof of the implication can be found in Ward's paper~\cite{wardexp} as well as in section~22 of Schmidt's book~\cite{schmidt}.

\begin{corol} \label{equid}
For $N \in \N$, let $\mu_{N}$ be the push-forward on $\mathcal{M}(X)$ of the Haar measure of the subgroup $\mathrm{Per}_N(T)$ of $N \cdot \Z^d$-periodic points of $(X,T)$. If $\nu \in \mathcal{M}(X)$ is a probability measure having $\mu_{N_i} \to \nu$ in the weak-$*$ topology for some sequence $N_i \to \infty$, then $\nu$ is a measure of maximum entropy: $h(\nu) = h(T)$.

In particular, if the system $(X,T)$ has a completely positive entropy, then the measures $\mu_{N}$ converge to the Haar measure of $X$: the periodic points equidisitribute in the Haar measure.
\end{corol}

One may wish to refine the equidistribution corollary to individual periodic trajectories. However, at least for rank $d = 1$, some care is needed in such a statement. Already for the case of the doubling map on the circle $\R/\Z$, it is not true that individual periodic trajectories are equidistributed in Haar measure as  $\langle \Gamma \rangle$ approaches infinity. This is similar to Duke's theorem, where an individual closed geodesic may be distributed in a practically arbitrary way. The parallel with Duke's theorem is discussed in Einsiedler, Lindenstrauss, Michel and Venkatesh~\cite{individual}, and appears to run rather deep. It includes the idea of equidistribution emerging from a lower bound on total length in terms of the discriminant of the quadratic order; respectively, what could be called following these authors the  ``discriminant'' (or ``denominator'') $\mathrm{disc}(P)$ of the periodic orbit. We shall not define this quantity here, but refer to section~1.7 of~\cite{individual} or to section~21 (``The discriminant-regulator paradox'') of McMullen~\cite{ctm} for the basic case of the doubling map of the circle, where the exponential sum bounds of Bourgain, Glibichuk and Konyagin~\cite{bgk,bourgain} are interpreted as giving equidistribution of individual periodic orbits having length $N$ exceeding $\mathrm{disc}(P)^{\epsilon}$ for a fixed $\epsilon > 0$. This refined equidistribution conjecture now extends to arbitrary Noetherian algebraic $\Z^d$-actions of finite entropy. In this generality, equidistribution appears to be a very difficult open problem, as is the corresponding conjecture in Duke's theorem and the higher rank generalizations thereof. A partial progress in the case of torus actions was made by Aka and Einsiedler~\cite{akaeinsiedler}.

We will focus our exposition on (B) and suppress the details of its equivalence with (A), and the consequence on equidistribution. Those details are well known in the $\Z^d$-actions community, and may be read, respectively, from the introduction following Theorem~1.3 of~\cite{atoral}, and from Ward's paper~\cite{wardexp} (exposed in chapter~22 of Schmidt's book~\cite{schmidt}).
Our proof of Theorem~\ref{convergence} is effective, in the sense of yielding a quantitative estimate on the convergence in (B) in terms of $d$, $N$ and the degree and height of the polynomial $P$.  Correspondingly, the equidistribution in the Corollary is also effective.

The convergence  problem (B) was also raised as Conjecture~15 in a recent paper~\cite{silverman} of Silverman, whose motivation was a study of higher rank divisibility sequences having a linear torus as the underlying algebraic group.
The cubical case of $\mu_N^d$ solved by Theorem~\ref{convergence} was given a separate attention as Conjecture~2 in Silverman's paper.

\subsection*{Acknowledgements}  The author is grateful to P. Habegger, D. Lind, H. Oh and T. Ward for their comments on the initial version of the paper,   P. Sarnak for pointing out the closely related problem of torsion growth in the homology of abelian covers, and A. Goncharov for his encouragement and advice throughout this project.

\section{Diophantine approximation by torsion points}  \label{statement}

By the classical Koksma-Hlawka inequality on numerical integration, Theorem~\ref{convergence} would follow at once if we could show
\begin{equation} \label{strong}
-\log{|P(\boldsymbol{\zeta}_N)|} = o_{N \to \infty}(N), \quad \textrm{for} \quad \boldsymbol{\zeta}_N \in \mu_N^d \, \textrm{ and } \, P(\boldsymbol{\zeta}_N) \neq 0.
\end{equation}
 This is not what we do, however, and~(\ref{strong}) remains wide open, as  does (B) for arbitrary sublattices $\Gamma \subset \Z^d$. It is important here to stress that (B) and (\ref{strong}) are both arithmetic statements, in that they could only be true for polynomials $P \in \bar{\Q}[x_1^{\pm},\ldots,x_d^{\pm}]$ having algebraic coefficients.
For polynomials with complex transcendental coefficients,   an exceptional set of $\boldsymbol{\zeta}$ is clearly needed in~(\ref{strong}).

Taking  account of such an exceptional set, a result of such a type does indeed hold, as Habbeger shows in a recently released manuscript~\cite{habegger}. Working in the much wider context of rational approximations to a subset of $\R^d$ definable in a fixed polynomially bounded $o$-minimal expansion of the real numbers, Habegger accomplishes this by continuing the determinental method of Bombieri, Pila and Wilkie; technically speaking, he does an extrapolation with an auxiliary function, and not an interpolation determinant. This continues an extensive literature on limiting the number of rational points lying in a definable set, which Habegger extends to rational points lying very near to the set. An overview of the method,
   in increasing order of technical detail, is presented  in Scanlon's paper~\cite{scanlon}, Zannier's monograph~\cite{zannier}, and the  recently released volume~\cite{ominimal} edited by Jones and Wilkie.

   Habegger proves a  general estimate of the same quality as the Pila-Wilkie counting theorem, in which closeness to the definable set is measured in terms of a Diophantine exponent depending on the definable set as well as the ``$\epsilon$'' of the counting theorem. The following result is implicit in Habegger's manuscript. It follows from his Theorem~2, using Ax's theorem~\cite{ax} for the description of the algebraic part of $P(e^{2\pi \sqrt{-1} x_1},\ldots,e^{2\pi \sqrt{-1} x_d}) = 0$.

Recalling the standard terminology, a \emph{torus coset} in $\mathbb{G}_m^d$ is a translate $\boldsymbol{\xi}T$ of a connected algebraic subgroup $T$ of $\mathbb{G}_m^d$ (an algebraic subtorus). The torus coset is called \emph{proper} if it is strictly contained in $\mathbb{G}_m^d$, i.e. $\dim{T} < d$. A \emph{torsion coset} is a torus coset $\boldsymbol{\xi} T$ having $\boldsymbol{\xi} \in \mu_{\infty}^d$ a torsion point. For a closed subvariety $X \subset \mathbb{G}_{m/\C}^d$ we write $X^{\circ}$ for $X$ minus the union of all positive dimensional torus cosets contained in $X$. This is a Zariski-open subset of $X$ (possibly empty), see for instance Theorem~4.2.3 (a) in Bombieri and Gubler's book~\cite{bg}. Embedding $\mathbb{G}_m^d \hookrightarrow \proj^d$ in the standard way, we denote by $\mathrm{dist}(\cdot,\cdot)$  the spherical distance on $\proj^d(\C)$ and use the convention that the distance to the empty set is zero.

 \begin{thm}[Habegger~\cite{habegger}]  \label{hab}
 Fix an $\epsilon > 0$ and a complex subvariety $X \subset \mathbb{G}_{m/\C}^d$.  Then, for every $N \gg_{X,\epsilon} 1$, all but at most $N^{\epsilon}$ points $\boldsymbol{\zeta} \in \mu_n^d$ having a finite order $n \leq N$ fulfil either
 \begin{equation} \label{strongish}
 - \log{\mathrm{dist}(\boldsymbol{\zeta},X)} \ll_{X,\epsilon} \log{N} \quad \textrm{or} \quad -\log{\mathrm{dist}(\boldsymbol{\zeta}, X \setminus X^{\circ})} \gg_{X,\epsilon} \log{N},
  \end{equation}
  with some implied constants depending on $X$ and $\epsilon$.
 \end{thm}

Compared to this we give the following result, which is limited to the $\bar{\Q}$-case and is stronger regarding the number of exceptions of a fixed order $N$. In contrast, it is less precise regarding bounded orders $n \leq N$ and the quality of the bound itself. We only consider an $o(\phi(N))$  bound as it is precisely the ``sub-Liouville'' quality of it that matters in our application to Theorem~\ref{convergence}.

For a non-zero multivariate Laurent polynomial $P \in \Z[x_1^{\pm 1}, \ldots, x_d^{\pm 1}]$, let $h(P)$ be the logarithm of the maximum absolute value of  a coefficient of $P$.  We take the degree $\deg{P}$ under the standard embedding $\mathbb{G}_m^d \hookrightarrow \proj^d$. We write $\phi(N) := [\Q(\mu_N):\Q] < N$ for Euler's function.

\begin{thm}  \label{main}
Fix an $\varepsilon > 0$ and an  integer Laurent polynomial
$$
P \in \Z[x_1^{\pm 1}, \ldots, x_d^{\pm 1}]
$$
whose zero locus $\{P = 0 \} \subset \mathbb{G}_m^d$ has each of its components mapping onto under all surjective homomorphisms $\mathbb{G}_m^d \twoheadrightarrow \mathbb{G}_m^{d-1}$. Then there are effectively computable functions
$$
M(d,\varepsilon,\deg{P}) < \infty \quad \textrm{and} \quad \quad N_0(d,\varepsilon,\deg{P},h(P)) < \infty
$$
and a finite union $Z = Z(d,\varepsilon,\deg{P})$ of proper torsion cosets $\boldsymbol{\mu}T \subsetneq \mathbb{G}_m^d$,
 such that the following is true:

  For every $N \geq N_0$,  all but at most $M(d,\varepsilon,\deg{P})$  points $\boldsymbol{\zeta} \in \mu_N^d$ satisfy either
\begin{equation} \label{subliouville}
 -\log{|P(\boldsymbol{\zeta})|} \leq \varepsilon \phi(N) \quad \textrm{ or } \quad P \in Z.
\end{equation}

\smallskip

The number of torsion cosets $\boldsymbol{\mu}T$ in $Z$, as well as the degrees of the subtori $T$ and orders of the torsion points $\boldsymbol{\mu}$, are also bounded by an effectively computable function of $d,\varepsilon$ and $\deg{P}$.
\end{thm}

The restriction to $\Z$-coefficients as opposed to $\bar{\Q}$-coefficients is no loss of generality.  Liouville's ``trivial'' Diophantine bound  gives (\ref{subliouville}) with some $\varepsilon < \infty$, depending on $d, \deg{P}$ and $h(P)$. Thus~(\ref{subliouville}) amounts to the subexponential improvement over Liouville's bound. In that regard it is a typical Diophantine statement, although we conjecture that the conclusion extends to complex polynomials.

By induction on the dimension, we derive the following version for subvarieties, similar to Habegger's Theorem~\ref{hab}.

\begin{thm} \label{subvar}
Fix an $\varepsilon > 0$ and a subvariety $X \subset \mathbb{G}_{m/\C}^d$ that can be defined over a number field. Then there is a finite constant $M = M(d,\varepsilon,\deg{X}) < \infty$ such that, for all $N \gg_{X,
\varepsilon} 1$, all but at most $M$ points $\boldsymbol{\zeta} \in \mu_N^d$ fulfil either
 \begin{equation} \label{weakish}
 - \log{\mathrm{dist}(\boldsymbol{\zeta},X)} \leq \varepsilon \phi(N) \quad \textrm{or} \quad -\log{\mathrm{dist}(\boldsymbol{\zeta}, X \setminus X^{\circ})} \geq \varepsilon \phi(N).
  \end{equation}
\end{thm}

We spell out the special case of Theorem~\ref{subvar} that $X$ is the hypersurface $1+x_1+\cdots+x_d = 0$. As is easily seen, the locus $X \setminus X^{\circ}$ is in this case given by the points of $X$ defined by the vanishing $x_{i_1} + \cdots + x_{i_r} = 0$ of a subsum consisting of $r \geq 2$ terms. Thus we get:

\begin{corol}  \label{corolsums}
 For every $\varepsilon > 0$ and $k \in \N$, there is a finite computable constant $C(k,\varepsilon)$ such that, for all $N$, all but at most $C(k,\varepsilon)$ sums $1+\zeta_1+ \cdots + \zeta_k$ with $\zeta_1, \ldots, \zeta_k \in \mu_N$ either  fulfil
 $$
 |1 + \zeta_1+ \cdots + \zeta_k| > e^{-\varepsilon \phi(N)},
$$
  or else have a proper subsum $\zeta_{i_1} + \cdots + \zeta_{i_r}$ with $|\zeta_{i_1} + \cdots + \zeta_{i_r}| < e^{-\varepsilon \phi(N)}$.
\end{corol}

    Habegger's and our methods can be described as doing Diophantine approximations, respectively, on the domain and codomain of the exponential (uniformization) map $\exp: \C^d \twoheadrightarrow \mathbb{G}_m^d(\C)$. Either  seems intrinsically incapable of reaching the other's result. In our case, this is because we work from the product formula in the level $N$ cyclotomic field,  doing Diophantine approximations on a power of the ambient linear torus where we need the torsion points to lie in a common cyclotomic field. In Habegger's case, it is due to working with the theory of the exponential function in a general context of rational approximation to sets definable in a polynomially bounded $o$-minimal expansion of the reals, where the transposition of our result towards~(\ref{strong}) does not hold.

If a point $\boldsymbol{\zeta}$ of order $N$ lies exactly on the hypersurface ($P(\boldsymbol{\zeta}) = 0$), then so does its Galois orbit, whose size $\phi(N) \to \infty$. Thus,  Theorem~\ref{main} is a generalization of the
 classical structural theorem about torsion points satisfying a polynomial relation, conjectured by Lang and first proved, independently, by Laurent~\cite{laurent} and Sarnak (in an unpublished manuscript, later revisited by Sarnak and Adams as~\cite{sarnakadams}):
  \begin{quote}
  The set of torsion points lying in a fixed proper algebraic subset $X \subsetneq \mathbb{G}_m^d$ over $\Q$ is covered by finitely many non-trivial monomial relations $\mathbf{x}^{\mathbf{n}} = 1, \mathbf{n} \neq \mathbf{0}$, and the number and degrees of these monomial relations is bounded by a certain explicit function of $d$ and the maximum degree in any set of polynomials equations defining $X$.
   \end{quote}
  This theorem, also known as ``toral Manin-Mumford,''  has  now  a multitude of different proofs and generalizations, and we refer to Zannier~\cite{zannier} for a thorough discussion. A particularly simple proof of the uniformity statement is in Bombieri and Zannier~\cite{bombierizannier}, also reproduced in chapter~4 of Bombieri and Gubler's book~\cite{bg}. For a sharper quantitative bound on the number of monomial relations, we refer to Aliev and Smyth~\cite{alievsmyth}.

 Theorem~\ref{main} does not however give a new proof of the Laurent-Sarnak theorem, which instead is used in a crucial point of the proof. The argument may be viewed as doing Diophantine approximation to extrapolate from a general structural result about points lying exactly on a subvariety to a general structural result about points lying very near to a subvariety.

An easy consequence of the Laurent-Sarnak theorem is that, for a fixed non-zero Laurent polynomial~$P$ and $\Gamma \subset \Z^d$ running over the finite index subgroups, the number of $\boldsymbol{\zeta} \in \mathbb{G}_m^d[\Gamma] = \{ \mathbf{x}^{\Gamma} = 1\}$ having $P(\boldsymbol{\zeta}) = 0$ is $o(|\mathbb{G}_m^d[\Gamma]|)$, effectively, as $\langle \Gamma \rangle \to \infty$.  This means that a negligible proportion of points is excluded by the sum in (B). Taking now $\Gamma = N \cdot \Z^d$, the Koksma-Hlawka inequality shows that the truncation of the sum to the complement of the $e^{-\varepsilon N}$-tubular neighborhood of the hypersurface $\{P = 0\}$ converges to the integral as $\varepsilon \to 0$ and then $N \to \infty$. A simple estimate shows that not more than $O(N^{d-1})$ of the $N$-torsion points are left behind in the tubular neighborhood (indeed, this even holds for the $1/N$-tubular neighborhood), and hence, applying the trivial Liouville bound on each of the exceptional points, Theorem~\ref{convergence} follows for $d \geq 2$ from  Theorem~\ref{main}. We give the details in the short section~\ref{hlawka} below. With a bit of additional work to treat the case that the zero locus $X$ of $P$ satisfies $X^{\circ} = \emptyset$, the same conclusion can be  obtained by Habegger's Theorem~\ref{hab}, with $\epsilon := 1/2$.
This leads to two independent proofs of Theorem~\ref{convergence} and Silverman's Conjecture~2 in~\cite{silverman}, one of which is treated in full detail by the present paper.

One may wonder whether or not Theorem~\ref{main} should in fact hold for all $\boldsymbol{\zeta} \in \mu_N^d$ having $P(\boldsymbol{\zeta}) \neq 0$, with an empty exceptional set. In this regard it may be worth remarking that Theorem~\ref{main} and its proof extends to a boundedness result on Diophantine approximation by small points of $\mathbb{G}_m^d$, and an exceptional set is definitely required in this generality. See section~\ref{complements} below. The necessity of an exceptional set is well known to be a salient feature already of the uniform Bogomolov problem, as exemplified by Theorem~4.2.3 with remark~4.2.6 in Bombieri and Gubler's book~\cite{bg}, or by  Theorem~1.3  of David and Philippon~\cite{davidphilippon}.

Our method for Theorem~\ref{main} will be that of Thue, Siegel, Roth and Schmidt. The novel point in this context is a use of Cartesian sets for the non-vanishing of the auxiliary construction at the special points.
Of course, non-vanishing at some point of a Cartesian product can be regarded as one of the oldest and simplest zero estimates, going back to Lang's multidimensional extension of the Schneider-Lang theorem. (A genuinely multidimensional Schneider-Lang theorem was subsequently achieved by Bombieri~\cite{bombierival}, without involving a Cartesian hypothesis.) It is also the basis of the polynomial method of Algebraic Combinatorics; see Alon~\cite{alon}.

Nonetheless, this simple idea appears not to have been exploited in the context of a Thue-Siegel method. Such a way of getting non-vanishing of the auxiliary construction at some special point, without using derivations, could also be used for Diophantine approximations over positive characteristic, and leads to a  substitute of Schmidt's Subspace theorem over a global function field even in the case that the divisor (or the linear forms) is not defined over the field of constants. The case of a constant divisor was solved by Julie Wang~\cite{wang} (we discuss this in relation to mixing in the final section~\ref{appo}), whereas for non-constant divisors the Vojta conjectures would appear to fail beyond hope over function fields of positive characteristic. We suggest that a finiteness statement should hold even then, concerning the best possible bound on clusters of sufficiently many solutions with comparable heights. This would be non-trivial in the higher dimensional case, as opposed to the classical situation of Roth's theorem. We refer to Kim, Thakur and Voloch~\cite{thakur} for a discussion of the one-dimensional case.

\smallskip

We hope to devote a subsequent paper to  a weak Subspace theorem over function fields of positive characteristic.

\medskip

{\it An outline, and organization.}
For inductive reasons, it turns out convenient to prove the thesis of Theorem~\ref{main} step by step, inductively from $k = 0$ to $k = d$, by proving the existence of  effectively computable functions $M_k = M_k(d,\varepsilon,\deg{P}) < \infty$ and $N_0 = N_0(d,\varepsilon,\deg{P},h(P)) < \infty$, and a finite union $Z_k = Z_k(d,\varepsilon,\deg{P}) \subsetneq \mathbb{G}_m^d$ of proper torsion cosets $\boldsymbol{\mu}T \subsetneq \mathbb{G}_m^d$, such that $|S| \leq M_k$ holds for
 any set $S \subset \mu_N^d$ of $N$-torsion points, $N \geq N_0$, that satisfies the following: $S \cap Z_k = \emptyset$,
\begin{equation} \label{definin}
 -\log{|P(\boldsymbol{\zeta})|} \geq \varepsilon \phi(N) \quad \textrm{for all } \boldsymbol{\zeta} \in S,
\end{equation}
 and  $S$ maps to a single point under the projection $\mathbb{G}_m^d = \mathbb{G}_m^k \times \mathbb{G}_m^{d-k} \to \mathbb{G}_m^{d-k}$ onto the last $d-k$ coordinates vector. This holds vacuously for $k = 0$, with $M_0 = 1$ and $Z_0 = \emptyset$, and we aim to prove the statement for the maximal (unrestricted) case $k = d$.

 Applying induction on $k$, we assume that $0 < k < d$ and that the statement is true for lower $k$ and all pairs $(P,\varepsilon)$.

 Let us write $\mathbf{x} = (\mathbf{y},\mathbf{z})$, where $\mathbf{y} = (y_1,\ldots,y_{k})$ stands for the first $k$ coordinates and $\mathbf{z} = (z_1,\ldots,z_{d-k})$ stands for the last $d-k$ coordinates.
Introducing a parameter $m$, to be taken sufficiently large with respect to $P$ and $\varepsilon$, and then a further large parameter $D \gg_m 1$, the idea is to create a polynomial identity in $\mathbf{z}$ and in the $m$ blocks $k$ of variables $\mathbf{y}_j := (y_{1,j},\ldots,y_{k,j})$, $j = 1,\ldots,m$:
\begin{eqnarray} \label{lhs}
F := \sum c_{\mathbf{i}_1,\ldots,\mathbf{i}_m,\mathbf{j}} \, \mathbf{y}_1^{\mathbf{i}_1} \cdots \mathbf{y}_m^{\mathbf{i}_m} \mathbf{z}^{\mathbf{j}} \\  = \sum_{r_1 + \cdots + r_m = n} \prod_{j=1}^m P(\mathbf{y}_j,\mathbf{z})^{r_j} \cdot Q_{\mathbf{r}}(\mathbf{y}_1,\ldots,\mathbf{y}_m,\mathbf{z}), \label{rhs}
\end{eqnarray}
  with  partial degrees less than $D$ in all $x_{i,j}$ and $z_s$, with order of vanishing at least
  \begin{equation} \label{order}
  n \geq \frac{mD}{C}, \quad \quad \quad \quad C = C(P) < \infty
  \end{equation}
   along the variety $P(\mathbf{y}_1,\mathbf{z}) = \cdots = P(\mathbf{y}_m,\mathbf{z}) = 0$, and with integral coefficients $c_{\mathbf{i}_1,\ldots,\mathbf{i}_m,\mathbf{j}} \in \Z$ having absolute values bounded by $e^{o(Dm)}$ as $m \to \infty$.

    Notice that simply raising $P$ to a power leads to coefficients with exponential size. The crucial subexponential estimate is obtained by a typical application of the Thue-Siegel lemma for a linear system in the unknown coefficients $c_{\mathbf{i}_1,\ldots,\mathbf{i}_m,\mathbf{j}}$, an application that  requires having $m \to \infty$ as $\varepsilon \to 0$. This is carried out in section~\ref{construction}, by separating a variable. We  note that in the basic  case of $m = 1$ and a single variable polynomial $P$, and most particularly for $P(x) = x-1$, much more refined constructions via this technique appear in the literature, see Bombieri and Vaaler~\cite{bombierivaaler}.

With this construction at hand,  consider now an $m$-tuple
$$
(\boldsymbol{\zeta}_1,\ldots,\boldsymbol{\zeta}_m) \in S \times \cdots \times S
$$
 of $N$-torsion points in $S$. Let $G_N$ denote the Galois group of the level $N$ cyclotomic field $C_N$.
 By assumption, all these points $\boldsymbol{\zeta}_j$ have a common vector of last $d-k$ coordinates. We wish  to specialize the right-hand side~(\ref{rhs}) of the identity to   $\mathbf{x}_j = (\mathbf{y}_j,\mathbf{z}) = \boldsymbol{\zeta}_j$ for  $\sigma = 1$, and the left-hand side~(\ref{lhs}) to all $\mathbf{x}_j = \boldsymbol{\zeta}_j^{\sigma}$ for $\sigma \neq 1$. If $F(\boldsymbol{\zeta}_1,\ldots,\boldsymbol{\zeta}_m) \neq 0$, then $\prod_{\sigma \in G_N} F(\boldsymbol{\zeta}_1^{\sigma},\ldots,\boldsymbol{\zeta}_m^{\sigma}) = N_{C_N/\Q}(F(\boldsymbol{\zeta}_1,\ldots,\boldsymbol{\zeta}_m)) \in \Z$ is a non-zero rational integer, hence at least~$1$ in absolute value. On taking $m,D \gg_{C,\varepsilon} 1$ and then  $N \gg_{F,\varepsilon} 1$ this will imply
   \begin{equation} \label{des}
   -\log{|P(\boldsymbol{\zeta}_j)|} \leq \varepsilon |G_N| = \varepsilon \phi(N)
   \end{equation}
   for at least one $j$, contradicting the defining assumption of the set $S$.

Suppose then that all points of $S \times \cdots \times S$ are constrained by the non-zero algebraic relation $F(\boldsymbol{\zeta}_1,\ldots,\boldsymbol{\zeta}_m) = 0$, that depends on $d, \varepsilon$ and the original $P$, but not on $N$. By the Laurent-Sarnak structural theorem (``Lang's $\mathbb{G}_m$ conjecture''), the set of all such $m$-tuples is covered by finitely many non-trivial multiplicative relations $\prod_{j=1}^m \boldsymbol{\zeta_j}^{\mathbf{r}_j} = 1$. Those relations that do not depend on the $\mathbf{y}_j$ variables must involve the $\mathbf{z}$ variable with a non-zero multi-index in $\Z^{d-k} \setminus \{0\}$. On enlarging $Z_{k-1}$ to a $Z_k$ by adding all these multiplicative $\mathbf{z}$ relations, we may assume that no point of $S$ fulfils any of the $(\mathbf{y}_1,\ldots,\mathbf{y}_m)$-independent relations. Then all points of the Cartesian power $S \times \cdots \times S$ are covered by finitely many fixed multiplicative relations, each of which involves at least one of the $y$-variables.
A positive proportion of points in $S \times \cdots \times S$ will belong to a fixed such relation, and a monoidal change of the $\mathbf{y}$-coordinates brings us easily within the scope of the induction hypothesis.

\section{The auxiliary construction}  \label{construction}

 It will be convenient to consider the notion of size on a logarithmic scale. Hence, for non-zero $P \in \Z[x_1,\ldots,x_d]$ (we may assume for Theorem~\ref{main} that $P$ is a polynomial), we write $h(P)$ for the maximum of the logarithms of absolute values of coefficients of $P$.

We write $\mathbf{y}_j$ for the length $k$  variables
 block $(y_{1j},\ldots,y_{kj})$ and $\mathbf{z}$ for the length $d-k$ variables $(z_1,\ldots,z_{d-k})$ and consider the ideal $\mathcal{I}_m \subset \Z[\mathbf{y}_1,\ldots,\mathbf{y}_m,\mathbf{z}]$ generated by the polynomials $P(\mathbf{y}_1,\mathbf{z}),\ldots,P(\mathbf{y}_m,\mathbf{z})$. The goal of this section is the following

\begin{propo} \label{poly} Assume that $k \in \{1,\ldots,d\}$ and that all non-constant factors of $P(\mathbf{y},\mathbf{z})$ are non-constant in $y_1$. Then there  is a constant $C = C(P) < \infty$, depending only on our polynomial $P$, such that the following is true. For all  $m \in \N$ and $D \gg_{m} 1$, there is a non-zero integer polynomial $F \in \Z[\mathbf{y}_1,\ldots,\mathbf{y}_m,\mathbf{z}]$ having all partial degrees $\deg_{y_{ij}}(F), \, \deg_{z_{s}}(F) < D$  and:
\begin{itemize}
\item[(i)] $\log{h(F)} \leq - m + \log{(CmD)}$;
\item[(ii)] for some $n \geq mD/C$, the polynomial $F$ belongs to the ideal $\mathcal{I}_m^n$.
\end{itemize}
\end{propo}

\medskip

We will need only the simplest possible version of Siegel's lemma; this is Lemma~2.9.1 in~\cite{bg}.

{\bf Siegel's lemma. } {\it Let $A$ be an $K \times L$ matrix, $L > K$, whose entries are rational integers bounded in magnitude by $B$. Then the homogeneous linear system $A \cdot \mathbf{x} = 0$ has a non-zero solution $\mathbf{x} = (x_1, \ldots, x_L) \in \Z^L$, satisfying
$$
h(\mathbf{x}) :=  \max_i{\log{|x_i|}} \leq \frac{K}{L-K} \log{L} + \frac{K}{L-K} \log{B}.
$$ \proofend}

\medskip

The next type of lemma,  of  probabilistic nature, is familiar from the parameter count in the proof of Roth's theorem (cf.~Bombieri-Gubler~\cite{bg}, Lemma~6.3.5).

\begin{lemma} \label{volume} The polytope
$$
E_m(A) :  \quad \{ \mathbf{t} \mid t_1 + \cdots + t_m < m/A, \quad 0 \leq t_1,\ldots,t_m \leq 1  \} \subset \R^{m}
$$
has volume satisfying
$$
\mathrm{vol}\big(E_m(A) \big) < (e/A)^m.
$$
\end{lemma}

{\it Proof. }   For any $\lambda > 0$, the characteristic function of $\{ t < 0\}$ is bounded by $e^{-\lambda t}$; let us choose $\lambda = A$. Since the condition $t_1 + \cdots + t_m < m/A$ takes the expression $t_1'+\cdots+t_m' + m/A < 0$ in the coordinates $t_i':=t_i-2/A$, we see that $\mathrm{vol}(E_m)$ is bounded by the integral of $\exp(-A(t_1' + \cdots + t_m'+m/A))$ over the box $\mathbf{t}' \in [-2/A,1]^m$. This is just the $m$-th power of the one-dimensional integral
$$
\int_{-\frac{2}{A}}^{1} e^{-A(t+1/A)} \, dt = \frac{1}{eA} \int_{-2}^A e^{-u} \, du < \frac{e}{A},
$$
proving our claim.
\proofend

\medskip

The construction now proceeds as follows. By our assumption that $k \geq 1$ and all non-constant factors of $P(\mathbf{y},\mathbf{z})$ are non-constant in $y_1$ we have, writing $\mathbf{y}' := (y_2,\ldots,y_k)$:
\begin{equation} \label{sepvar}
P(\mathbf{x}) = R(\mathbf{y}',\mathbf{z}) y_{1}^r - \sum_{s=0}^{r-1} R_s(\mathbf{y}',\mathbf{z})y_1^s, \quad R, R_s \in \Z[\mathbf{y}',\mathbf{z}],
\end{equation}
where $R \neq 0$ and the $R_s$ are integer  polynomials without  common non-constant factor.

  We look for a  polynomial
   $$
   F(\mathbf{y}_1,\ldots,\mathbf{y}_m,\mathbf{z}) = \sum c_{\mathbf{i}_1,\ldots,\mathbf{i}_m,\mathbf{j}} \, \mathbf{y}_1^{\mathbf{i}_1}\cdots \mathbf{y}_m^{\mathbf{i}_m} \mathbf{z}^{\mathbf{j}} \in \Z[\mathbf{y}_1,\ldots,\mathbf{y}_m,\mathbf{z}]
   $$
    in $km + (d-k)$ variables $y_{tj}$, $t = 1,\ldots,k, \, j = 1,\ldots,m$ and $z_1, \ldots, z_{d-k}$, in which all components  of the multi-indices $\mathbf{i}_s \in \N_0^k$ and $\mathbf{j} \in \N_0^{d-k}$  take values in $\{0,1,\ldots,D-1\}$.
      Then the polynomial
    \begin{equation} \label{asso}
  Q(\mathbf{y}_1,\ldots,\mathbf{y}_m,
\mathbf{z}) := R(\mathbf{y}_1',\mathbf{z})^{D} \cdots R(\mathbf{y}_m',\mathbf{z})^{D} F(\mathbf{y}_1,\ldots,\mathbf{y}_m,\mathbf{z})
    \end{equation}
    still has its partial degrees in the ``first variables'' $y_{1j}$ smaller  than $D$.

     Consider the following finite procedure applied iteratively starting from~$Q$:

      \begin{quote}
      {\it At each stage, every occurrence of an $y_1^r$ will be accompanied by a corresponding $R(\mathbf{y}',\mathbf{z})$. Replace $R(\mathbf{y}',\mathbf{z})y_d^r$ by $P(\mathbf{y},\mathbf{z}) + \sum_{s=0}^{r-1} R_s(\mathbf{y}',\mathbf{z})y_1^s$, and repeat until the partial degrees in the $y_1$ are all $< r$.}
      \end{quote}

     Owing to the exponential boundedness of linear recursions and multinomial coefficients, at the end of our procedure we have presented $Q$ in the form
     \begin{equation} \label{expan}
    Q = \sum_{j=1}^m \sum_{s_j=0}^{r-1} \sum_{r_j =0}^{D-1} \sum_{\mathfrak{I}}   \sum_{\mathbf{j}} l_{\mathfrak{I},\mathbf{j},\mathbf{r},\mathbf{s}} \mathbf{x}'^{\mathfrak{I}} \mathbf{z}^{\mathbf{j}} F(\mathbf{x}_1)^{r_1} \cdots F(\mathbf{x}_m)^{r_m}  y_{11}^{s_1} \cdots y_{1m}^{s_m},
     \end{equation}
     where the $ l_{\mathfrak{I},\mathfrak{j},\mathbf{r},\mathbf{s}}$ are linear forms in the $c_{\mathbf{i}_1,\ldots,\mathbf{i}_m,\mathbf{j}}$ with coefficients rational integers of logarithmic size $\ll_P mD$.
     Here, $\mathfrak{I} := (\mathbf{i}_1',\ldots,\mathbf{i}_m')$ ranges over $m$-tuples $\mathbf{i}' \in \{0,1,\ldots, 2D-1\}^{k-1}$; $\mathbf{j}$ ranges over $\{0,1,\ldots,2D-1\}$; $\mathbf{x}_j$ as before stands for the block $(\mathbf{y}_j,\mathbf{z})$, and similarly $\mathbf{x}_j' := (\mathbf{y}_j',\mathbf{z})$; and  $\mathbf{x}'^{\mathfrak{I}}$ is an abbreviation for $\mathbf{x}_1'^{\mathbf{i}_1'} \cdots \mathbf{x}_m'^{\mathbf{i}_m'}$.

   \medskip

{\it Proof of Proposition~\ref{poly}. } Clearly, we may assume the co-primality of the coefficients of $F$, and hence that $\mathrm{gcd}(R,R_0,\ldots,R_{r-1})$ is the unit ideal of $\Z[\mathbf{x}]$ in (\ref{sepvar}). We apply Siegel's lemma to  the linear system
\begin{equation} \label{siegeleq}
l_{\mathfrak{I},\mathbf{j},\mathbf{r},\mathbf{s}} = 0, \quad r_1 + \cdots + r_m < n
\end{equation}
in the unknowns $c_{\mathbf{i}_1,\ldots,\mathbf{i}_m,\mathbf{j}}$. As each $\mathbf{i}_t$ ranges over $\{0,\ldots,D-1\}^k$ and $\mathbf{j}$ ranges over $\{0,\ldots,D-1\}^{d-k}$, the number of free parameters is $L := D^{km + (d-k)}$. We use Lemma~2 to estimate the number $K$ of equations~(\ref{siegeleq}). We will take $n := mD/C + O(1)$ with a suitable sufficiently large constant $C = C(P) < \infty$ depending only on $P$. Then, in the asymptotic that $D \to \infty$ ($m$ is fixed), the number of lattice points in the dilation by $D$ of the polytope $E_m(C)$ is $\sim D^m \mathrm{vol} \big( E_m(C) \big) < (e/C)^m D^m$. The multi-index $\mathfrak{J}$ takes $(2D)^{(k-1)m}$ possibilities, $\mathbf{j}$ takes $(2D)^{d-k}$ possibilities, and $\mathbf{r} = (r_1,\ldots,r_m)$ takes $r^m$ possibilities. We conclude that for $D \gg_{m,C} 1$ large, there are fewer than
$$
K < (er\cdot 2^{d-1}/C)^m D^{km + (d-k)} = (er\cdot 2^{d-1}/C)^m L
$$
equations to solve.

We choose
\begin{equation} \label{firstchc}
C \geq  e^2r \cdot 2^{d-1},
\end{equation}
making
\begin{equation} \label{equations}
K < e^{-m}L.
\end{equation}
The Dirichlet exponent of our linear system is $K/(L-K) < 2e^{-m}$, while the logarithmic size $\log{B}$ is bounded by $\ll_P mD$. Consequently, Siegel's lemma supplies a non-zero solution $F \in \Z[\mathbf{y}_1,\ldots,\mathbf{y}_m,\mathbf{z}]$, of logarithmic size $h(F)$ satisfying
$$
\log{h(F)} \leq -m + \log(mD) + O_P(1),
$$
such that~(\ref{siegeleq}) hold in~(\ref{expan}) with $n \geq md/C$. On augmenting the choice of $C$ in~(\ref{firstchc}) to subsume the $O_P(1)$ constant, we gain~(i).

For (ii), we have insured that the associated polynomial $Q$ in (\ref{asso}) lies in $\mathcal{I}_m^n$. It remains to prove that this in fact implies the {\it a priori} stronger conclusion $F \in \mathcal{I}_m^n$. For this, noting the co-primality of $R$ with the $R_s$ in~(\ref{sepvar}) and that none of these polynomials involves $y_1$, it suffices to remark that none of the $R(\mathbf{y}_j',\mathbf{z})$  belong to any of the associated primes of the ideal $\mathcal{I}_m$. \proofend

\section{Non-vanishing of the auxiliary construction \\ at some special point}

We fix an $\varepsilon > 0$ and set out to prove by induction on $k \in \{0,\ldots,d\}$ the successively stronger statements that there exists a finite union  $Z_k(P,\varepsilon)$ of non-trivial multiplicative relations $\mathbf{x}^{\mathbf{n}} = 1$, $\mathbf{n} \neq \mathbf{0}$, whose number and degrees $n_1+\cdots+n_d$ are bounded by a computable function of $d,\varepsilon$, and $\deg{P}$ alone, such that the following is true for every $N \gg_{k,d,\varepsilon,\deg{P},h(P)} 1$:
\begin{quote}
  If all elements of a set $S \subset \mu_N^d$ of $N$-torsion points satisfy (\ref{definin}) but not any of the multiplicative relations from $Z_k(P,\varepsilon)$, and share a common vector for their last $d-k$ coordinates, then $|S| \leq M_k(d,\varepsilon,\deg{P})$ is bounded by a computable function of $k,d,\varepsilon$ and $\deg{P}$.
  \end{quote}

This trivially holds for $k = 0$, with $M_0 = 1$ and $Z_0(P,\varepsilon) = \emptyset$. We proceed to the induction step, taking $k > 0$ and assuming the statement holds for lower values of $k$.
Choose $m$ in Proposition~\ref{poly} to be large enough subject to
\begin{equation}  \label{choicem}
e^{-m} < \varepsilon/(4C^2)
 \end{equation}
 and fix it. By our assumption that the components of $\{P = 0\} \subset \mathbb{G}_m^d$ map onto under all surjective homomorphisms $\mathbb{G}_m^d \twoheadrightarrow \mathbb{G}_m^{d-1}$, all non-constant factors of the polynomial $P(\mathbf{y},\mathbf{z})$ are non-constant in the variable $y_1$, and hence Proposition~\ref{poly} applies.  Fixing then a $D \gg_m 1$ as in the proposition, assumed for later reference to satisfy
  \begin{equation} \label{choiced}
  dm\log{D} < \frac{\varepsilon}{4C} mD,
  \end{equation}
  we obtain a non-zero polynomial $F \in \Z[\mathbf{y}_1,\ldots,\mathbf{y}_m,\mathbf{z}]$ of partial degrees bounded by $D$ and fulfilling (i) and (ii) of Proposition~\ref{poly}.

Assume first that  $F$ vanishes, for some $N$ and $S$ as above, at all points of the Cartesian power $S \times \cdots \times S$. The Laurent-Sarnak structural theorem then implies that each point of this Cartesian power fulfils one of finitely many non-zero monomial relations
$$
R: \quad \mathbf{y}_1^{\mathbf{n}_1} \cdots \mathbf{y}_m^{\mathbf{n}_m} \mathbf{z}^{\mathbf{t}} = 1, \quad (\mathbf{n}_1,\ldots,\mathbf{n}_m,\mathbf{t}) \in \Z^{km+(d-k)} \setminus \{\mathbf{0}\}.
$$
We will include in $Z_k(P,\varepsilon)$ at least the union $Z'$ of the finitely many non-zero multiplicative relations $\mathbf{z}^{\mathbf{t}} = 1$ (or $\mathbf{x}^{(\mathbf{0},\mathbf{t})} = 1$ in terms of $\mathbf{x} = (\mathbf{y},\mathbf{z})$)  from $R$ that have $\mathbf{n}_1 = \cdots = \mathbf{n}_m = \mathbf{0}$ (so necessarily $\mathbf{t} \neq \mathbf{0}$). This allows us to assume that the Cartesian power $S \times \cdots \times S$ is covered by finitely many relations
\begin{equation} \label{correl}
\mathbf{y}_1^{\mathbf{n}_1} \cdots \mathbf{y}_m^{\mathbf{n}_m} = \xi_0, \quad (\mathbf{n}_1,\ldots,\mathbf{n}_m) \in \Z^{km} \setminus \{\mathbf{0}\},
\end{equation}
 where $\xi_0 \in \mu_N$ are $N$-th roots of unity while $(\mathbf{n}_1, \ldots, \mathbf{n}_m)$ are non-zero vectors in $\Z^{km}$ that depend on $F$ and not on $N$. Moreover, the number and degrees of these relations can be taken to be bounded by a universal function $B(d,m,\deg{F})$.

 By the pigeonhole principle, one of the relations (\ref{correl}) is fulfilled by at least $|S|^m / B(d,m,\deg{F})$ of the $|S|^m$ points of the Cartesian power. If $\mathbf{y}_1^{\mathbf{n}_1} \cdots \mathbf{y}_m^{\mathbf{n}_m} = \xi_0$ is this relation and $j_0$ an index where $\mathbf{n} = \mathbf{n}_{j_0} \neq \mathbf{0}$, the pigeonhole principle again implies that for some point $\mathbf{p} \in \prod_{j \neq j_0} S$, the set of solutions to $\mathbf{y}_1^{\mathbf{n}_1} \cdots \mathbf{y}_m^{\mathbf{n}_m} = \xi_0$ in $S \times \cdots \times S$ projecting to $\mathbf{p}$ has cardinality at least $\lfloor |S| / B(d,m,\deg{F}) \rfloor$.  Specializing the $j \neq j_0$ coordinates to $\mathbf{p}$, this then gives the existence of another $N$-th root of unity, $\xi \in \mu_N$, such that the fixed non-zero multiplicative relation $\mathbf{y}^{\mathbf{n}} = \xi$ is fulfilled by the vectors of first $k$ coordinates $\mathbf{y}$ of all points in a subset $S' \in S$ of cardinality at least $|S'| \geq \lfloor |S| / B(d,m,\deg{F}) \rfloor$.

Upon relabeling the coordinates $y_1, \ldots, y_k$, we may assume $n_k \neq 0$ and let $u_1 := y_1, \ldots, u_{k-1} := y_{k-1}$ and $u_k := y_1^{n_1} \cdots y_k^{n_k}$. Then
\begin{eqnarray*}
G:= \prod_{\mu: \, \mu^{n_k} = 1} P(y_1, \ldots, y_{k-1},\mu y_k,\mathbf{z}) \\ \in \Z[y_1^{\pm 1}, \ldots,y_{k-1}^{\pm 1};y_k^{\pm n_k};z_1^{\pm 1}, \ldots, z_{d-k}^{\pm 1}]
\end{eqnarray*}
is a non-zero integer Laurent polynomial in  $y_1 = u_1, \ldots,  y_{k-1} = u_{k-1}$, $y_k^{n_k} = u_k \cdot u_1^{-n_1} \cdots u_{k-1}^{-n_{k-1}}$ and $\mathbf{z}$, and hence a non-zero integer Laurent polynomial in $u_1,\ldots,u_k$ and $\mathbf{z}$. By construction, all points $\boldsymbol{\zeta} \in S'$ satisfy $u_k = \xi$ and
\begin{eqnarray} \label{gee}
 -\log{|G(\boldsymbol{\zeta})|} \geq  -\log{|P(\boldsymbol{\zeta})|} + O_F(1) \geq \varepsilon \phi(N) /2,
\end{eqnarray}
for $N \gg_{k,d,\varepsilon,\deg{P},h(P)} 1$, with implied constant depending effectively on $k,d, \varepsilon,m,C$ and $D$, and hence ultimately on $k, d,\varepsilon, \deg{P}$ and $h(P)$.

 The points in $S'$ thus fulfil (\ref{gee}) and have a common vector for their last $d-k+1$ coordinates. Our assumption on the non-degeneracy of the components under homomorphisms $\mathbb{G}_m^d \twoheadrightarrow \mathbb{G}_m^{d-1}$ again implies that $G(\boldsymbol{u},\mathbf{z})$ has all its non-constant irreducible factors non-constant in the variable $u_1$.  Applying the induction hypothesis with $P$ replaced by the expression of $G$ in the $(\mathbf{u},\mathbf{z})$ coordinates, and with $\varepsilon/2$ in place of $\varepsilon$, we define $Z_k(P,\varepsilon) := Z_{k-1}(G,\varepsilon/2) \cup Z'$ and conclude that $|S'|$ is bounded by a computable function of $d,\varepsilon$ and $\deg{G}$, and hence by a computable function of $d,\varepsilon$ and $\deg{P}$. Hence the same is true for $|S| \leq 2B(d,m,\deg{F}) \cdot |S'| \ll_{d,\varepsilon,\deg{P}} |S'|$.

 This brings the vanishing assumption to the range of the induction hypothesis,
   and hence allows us to assume $F(\boldsymbol{\zeta}_1,\ldots,\boldsymbol{\zeta}_m) \neq 0$ for some $\boldsymbol{\zeta}_1, \ldots, \boldsymbol{\zeta}_m \in S$.

 \section{Proof of Theorems~\ref{main} and~\ref{subvar}}  \label{finalarg}

  For Theorem~\ref{main}, we are  reduced to proving that at least one of $\boldsymbol{\zeta}_1, \ldots, \boldsymbol{\zeta}_m \in \mu_N^d$ has to satisfy~(\ref{des}) if $F(\boldsymbol{\zeta}_1,\ldots,\boldsymbol{\zeta}_m) \neq 0$ and $N \gg_{P,\varepsilon} 1$. We argue by contradiction.

  Writing $G_N$ for the Galois group of $\Q(\boldsymbol{\zeta})/\Q$ and noting that the number of terms in $F$ is bounded by $D^{dm}$, clause (i) in Proposition~\ref{poly} along with our choices~(\ref{choicem}) and (\ref{choiced}) insure the bound
\begin{equation}  \label{upperb}
\log{|F(\boldsymbol{\zeta}^{\sigma})|} \leq dm\log{D} + e^{-m}CmD < \varepsilon \, \frac{mD}{2C}, \quad \textrm{ all } \sigma \in G_N.
\end{equation}
We  use this bound for $\sigma \neq 1$, while for $\sigma = 1$ we use instead the right-hand side of the identity~(\ref{rhs}) and the defining assumption that $-\log{|P(\boldsymbol{\zeta}_j)|} \leq \varepsilon |G_N|$ for all $j = 1,\ldots,m$:
\begin{equation} \label{sumdexc}
 -\log{|F(\boldsymbol{\zeta})|} \geq \varepsilon \frac{mD}{C} \phi(N) - O_{C,m,D}(1).
\end{equation}

We may now conclude the proof of Theorem~\ref{main}. Since $\prod_{\sigma \in G_N} F(\boldsymbol{\zeta}^{\sigma}) \in \Z \setminus \{0\}$ is a non-zero rational integer, we have $\sum_{\sigma \in G_N} - \log{|F(\boldsymbol{\zeta}^{\sigma})|} \leq 0$. Summing the negative of~(\ref{upperb}) over $\sigma \in G_N \setminus \{1\}$ and~(\ref{sumdexc}), we reach a contradiction as soon as $N \gg_{\varepsilon,C,m,D} 1$. \proofend

\smallskip

We now derive Theorem~\ref{subvar} from Theorem~\ref{main}. It suffices to prove the same statement for algebraic subsets $X \subset \mathbb{G}_m^d$ (over $\Z$). There is a finite set of non-zero integer polynomials $P_1, \ldots, P_r \in \Z[\mathbf{x}]$ having $X$ as their common zero locus. We apply a double induction, first on $d$ and then on $\sum_{i=1}^r \deg{P_i}$, with the base case $d = 1$  holding trivially, and the base case $\sum_{i=1}^r \deg{P_i} = 0$ being vacuous. For the induction step inside $\mathbb{G}_m^d$, suppose first that the polynomial $P_1$ meets the condition of Theorem~\ref{main}, of having all its components map onto under all surjective homomorphisms $\mathbb{G}_m^d \twoheadrightarrow \mathbb{G}_m^{d-1}$. In that case, as $-\log{\mathrm{dist}(\boldsymbol{\zeta},X)} \leq - \log{|P_1(\boldsymbol{\zeta})|} + O_{P_1,X}(1)$, we get the desired bound
$$
-\log{\mathrm{dist}(\boldsymbol{\zeta},X)} \leq \varepsilon \phi(N)
$$
for $\boldsymbol{\zeta} \in \mu_N^d$ apart from an $O_{P_1,\varepsilon}(1)$ number of exceptions (depending on $N$) and a finite union  of proper torsion cosets $\boldsymbol{\mu} T \subsetneq \mathbb{G}_m^d$ depending on $P_1$ and $\varepsilon$ but independent of $N$. We then apply the induction hypothesis on $d$ to each of these finitely many lower dimensional tori $T$ and the subvarieties $\boldsymbol{\mu}^{-1} X \cap T$ inside them.

Suppose now that the hypersurface $\{ P_1 = 0\} \subsetneq \mathbb{G}_m^d$ has an irreducible component $V$ that maps onto a hypersurface $U \subsetneq \mathbb{G}_m^{d-1}$ under a surjective homomorphism $p: \mathbb{G}_m^d \twoheadrightarrow \mathbb{G}_m^{d-1}$. The induction hypothesis on $\sum_{i=1}^r \deg{P_i}$ takes care of the points $\boldsymbol{\zeta}$ lying close to the algebraic set $X \setminus V$. But $V \subset X \setminus X^{\circ}$ as our assumption implies that $V$ is covered by a union of translates of one dimensional algebraic subgroups of $\mathbb{G}_m^d$, namely the translates by points in $p^{-1}U$ of the kernel of the homomorphism $p$. This completes the double induction and the proof of Theorem~\ref{subvar}. \proofend

\section{Proof of Theorem~\ref{convergence}}  \label{hlawka}

We appeal to the classical Koksma-Hlawka inequality, a standard reference for which is chapter~5 of Kuipers and Niederreiter's monograph~\cite{uniform}. The function $f_T(\mathbf{z}) := \max(-T,\log{|P(\mathbf{z})|})$ on the torus $(S^1)^d$ has Hardy-Krause total variation $O_{P}(T)$ as $T \to +\infty$, and the set (sequence) $\mathbb{G}_m^d[N] = \mu_N^d$ has discrepancy $O_d(1/N)$. The complement of any subset of $U \subset \mu_N^d$ of cardinality  $|U| = O_P(N^{d-1})$ still has discrepancy $O_P(1/N)$. As $\int_0^{e^{-T}} \log{t} \, dt =  (1-T)e^{-T}$ (integrability of the log singularity), the integrals of $f_T$ and $\log{|P|}$ over the torus differ by less than $1/T$ for $T \gg_P 1$ and so, by the Koksma-Hlawka inequality, all $T > 1$ and $\Gamma$ and $U$ as above satisfy
\begin{eqnarray*} \label{bulk}
\Big| \frac{1}{N^{d}} \sum_{\boldsymbol{\zeta} \in  \mathbb{G}_m^d[N] \setminus U}  \max(-T,\log{|P(\boldsymbol{\zeta})|})  - \int_{(S_1)^d} \log{|P|} \, d\boldsymbol{\theta}  \Big|
\\ \ll_P T/N + 1/T,
\end{eqnarray*}
where the implied constant depends only on the polynomial $P$.

We may clearly assume that $P$ is irreducible in $\Z[\mathbf{x}]$ and, upon decreasing $d$, that it satisfies the condition of Theorem~\ref{main}: the zero locus $\{P = 0\}$ maps onto under all surjective homomorphisms $\mathbb{G}_m^d \twoheadrightarrow \mathbb{G}_m^{d-1}$. Choose $T := \varepsilon N$ and $U$ the set of $\boldsymbol{\zeta} \in \mu_N^d$ with $\log{|P(\boldsymbol{\zeta})|} \leq -\varepsilon N$. Then the right hand side of the Koksma-Hlawka bound approaches zero as first $\varepsilon \to 0$ and then $N \to \infty$. We have to show two points under this asymptotics:
\begin{itemize}
\item[(1)] $|U| \ll_{d,P} N^{d-1}$ (uniformly in $\varepsilon$); and
\item[(2)] $\sum_{\substack{ \boldsymbol{\zeta} \in U \\ P(\boldsymbol{\zeta}) \neq 0 } } -\log{|P(\boldsymbol{\zeta})|} \ll_P \varepsilon N^d$.
\end{itemize}

In fact (1) even holds for the number of $\mu_N^d$-points in the $1/N$ tubular neighborhood of $\{P = 0\} \subset (S^1)^d$. For, the volume of the $\sqrt{d}/N$ tubular neighborhood is $O_{d,P}(1/N)$, and that tubular neighborhood contains the disjoint union of the $1/N \times \cdots \times 1/N$ hypercubes having for Southwest corner a $\mathbb{G}_m^d[N]$-point in the $1/N$-tubular neighborhood, using the identification $(S^1)^d = (\R/\Z)^d$. Similarly, as the exceptional set $Z$ in Theorem~\ref{main} is a union of finitely many lower dimensional torsion cosets independent of $N$, we have $|U \cap Z| \ll_{d,P,\varepsilon} N^{d-2}$.

Now (2) also
follows immediately for $d \geq 2$  from Theorem~\ref{main}, applying~(\ref{subliouville}) to each of the $\ll N^{d-1}$ points in $U$ lying in neither $Z$ nor the exceptional set of $O_{P,\varepsilon}(1)$ elements,
 and the trivial Liouville $\ll_P N$ bound at each of the remaining $\ll_{P,\varepsilon} 1 + N^{d-2}$ points.  This completes the proof of Theorem~\ref{convergence} in the $d \geq 2$ case, whereas the $d = 1$ case follows from Gelfond's theorem, see Lind~\cite{lindqh}.

 \smallskip

  The same conclusion in the $d \geq 2$ case can just as  well be obtained from Habegger's Theorem~\ref{hab}, with $\epsilon := 1/2$. For the argument just given, we only need to know that
   $X^{\circ} = \emptyset$ implies that the hypersurface $\{P = 0\}$ is degenerate, in the sense of having an irreducible component that does not map onto under a surjective homomorphism $\mathbb{G}_m^d \twoheadrightarrow \mathbb{G}_m^{d-1}$. This follows from Proposition~2.5 in Chambert-Loir~\cite{chambertloir}, for the open anomalous set  $X^{\mathrm{oa}}$ of \emph{loc.cit.} is contained in $X^{\circ}$.
   \proofend

\section{The integral homology of abelian congruence covers}

\subsection{Homology of congruence covers} Sarnak's proof~\cite{sarnakadams} of the ``Lang $\mathbb{G}_m$'' conjecture, which is essentially the special case of Theorem~\ref{main} for torsion points lying exactly on a subvariety, arose in the context of studying the behavior of Betti numbers as a function of the level $N$ in congruence Galois covers $M_N \to M$ of a fixed finite simplicial complex $M$, corresponding to the kernel of reduction modulo $N$ under a representation $\rho : \pi_1(M) \to H(\Z)$ of the fundamental group of $M$ in the group of $\mathbb{Z}$-points of a linear algebraic group $H$  that possesses the strong approximation property. When $H$ is either unipotent or semisimple, Adams and Sarnak~\cite{sarnakadams} prove that the Betti numbers $b_i(M_N) = \dim_{\Q} H_i(M_N,\Q)$, $i = 0,1,\ldots$, are piecewise polynomial functions of the level $N$. This may be viewed as a vast generalization of  the classical  polynomial formula $\binom{N-1}{2}$ for the genus of the Fermat curve $\{X^N + Y^N = Z^n\} \subset \proj^2$, and is related to the fact, itself an immediate consequence of the Laurent-Sarnak structural theorem, that the number of $N$-torsion points lying in a fixed algebraic subset of $\mathbb{G}_m^d$ is a piecewise polynomial function in $N$.

One would like to understand the asymptotic behavior of the full integral homology $H_i(M_N,\Z)$ of the congruence cover $M_N$, which is the direct sum of a free $\Z$-module of rank $b_i(M_N)$ (piecewise polynomial in $N$) and a finite abelian group $H_i(M_N,\Z)_{\mathrm{tors}}$. It remains to understand the asymptotics of the torsion subgroup. When non-periodic, the torsion group $H_i(M_N,\Z)_{\mathrm{tors}}$ seems always to have an exponential (positive) exact growth rate in the appropriate power of~$N$.  In particular, similarly to a theorem of L\"uck~\cite{luckconvergence} answering a question of Kazhdan and Gromov for the Betti numbers, and under the assumption that $M$ is acyclic,  the finite layer Reidemeister torsions
 for the congruence kernels of the representation $\rho$ are expected to converge to the $L^2$-torsion of $(M,\rho)$. We refer the reader to L\"uck's book~\cite{luck} for the definition and a detailed study of the $L^2$-versions of these and other classical invariants. Bergeron and Venkatesh~\cite{bergeronvenkatesh} have a precise conjecture for the case that $H$ is semisimple, and obtain results in the lower bound direction under a spectral gap condition, for strongly acyclic local systems.

\subsection{The abelian case: $H(\Z) = \Z^d$}

In this section, we combine Theorem~\ref{main} and work of Thang Le~\cite{le} to treat the case that $H$ is the vector group $\mathbb{G}_a^d$: the case of the congruence abelian covers. The formula for the Betti numbers in that case, given as a count of $N$-torsion points lying in a certain ``explicit'' algebraic subset of $\mathbb{G}_m^d$, is summarized in section~7.2 of Bergeron and Venkatesh's paper~\cite{bergeronvenkatesh}.  We may now complement this with:

\begin{thm} \label{abelian}
Let $M$ be a finite simplicial complex and $R: H_1(M,\Z) \twoheadrightarrow \Z^d$ a surjection.  Consider $M_N$ the $(\Z/N)^d$ Galois cover of $M$ corresoponding to the kernel of reduction modulo $N$ of $R$. Then, for each $i = 0,1,\ldots$, the limit
$$
\lim_{N \to \infty} N^{-d} \log{|H_i(M_N,\Z)|_{\mathrm{tors}}}
$$
exists and equals the (additive) Mahler measure $m(P)$ of  a certain non-zero integer Laurent polynomial $P \in \Z[x_1^{\pm 1}, \ldots, x_d^{\pm 1}] \setminus \{0\}$ in $d$ commuting variables.
\end{thm}

\begin{proof}
Follows immediately upon combining Theorem~\ref{convergence} with the proof of Theorem~5 of Thang Le~\cite{le}.
The polynomial $P$ is the first non-zero Alexander polynomial of $H_i(\widetilde{M}(R),\Z)$ as a module over the Laurent series ring $\Z[\Z^d] \cong R_d$ (if this module is non-zero), where $\widetilde{M}(R) \to M$ is the (infinite) abelian $\Z^d$-cover corresponding to the kernel of $R$. For the (standard) definition of the sequence $(\Delta_i)_{i \geq 0}$, $\Delta_{i+1} \mid \Delta_i$ of Alexander polynomials of a Noetherian $R_d$-module, we refer to section~1.2 of Le's paper~\cite{le}.
\end{proof}

\subsection{Link exteriors} The most interesting and classical case concerns the exterior of a link of $d$ components in $S^3$, or more generally in any integral homology $3$-sphere, and  its abelian $(\Z/N)^d$-covers as well as the corresponding branched coverings at the link. (The latter are compact $3$-manifolds, typically hyperbolic.) This is the context in which the Alexander polynomial was originally discovered, as the first historic example of a polynomial invariant of a knot  in $3$-space.

In the most basic case that $M$ is homeomorphic to a knot complement in $S^3$,  the formulas for the Betti numbers and the cardinal number of the first integral homology groups of the abelian cyclic covers of $M$ are classical results of, respectively, Goeritz and Fox;  see section~2.11 of Gordon~\cite{gordon}. However, in these results, the cardinal of an infinite group was interpreted to be zero. It took the work of Silver and Williams~\cite{silverwilliams} to isolate the torsion number in the case of a positive Betti number, which they did  by linking the problem to the connected components count in an algebraic dynamical system.

Silver and Williams used diagram  colorings to construct the system, and worked in the more general context of links and algebraic $\Z^d$-actions. They could then conclude (\cite{silverwilliams}, Theorem~2.1) that the \emph{upper} exponential growth rate of the size of the torsion $H_1(M_{\Gamma},\Z)_{\mathrm{tors}}$ in the first homology of the $\Z^d/\Gamma$-cover of $S^3$ branched at the link $l$ equals the Mahler measure of the first Alexander polynomial $\Delta_0$ of the link, as $\langle \Gamma \rangle \to \infty$, \emph{provided} $\Delta_0 \neq 0$. Instead, Thang Le's work~\cite{le} applied a comparison result to reduce to the case of a torsion $R_d$-module, and thereby succeeded to remove the non-vanishing $\Delta_0 \neq 0$.  As a consequence of Theorem~\ref{abelian} and the proofs in section~4 of Le's paper, we obtain a refinement of these results to the existence of the limit (Conjecture~6.1 of Silver and Williams~\cite{silverwilliams}) across all congruence abelian coverings.

\begin{corol} \label{links}
Let $l$ be an oriented link with $d$ components in an integral homology $3$-sphere $S$. We have $H_1(S \setminus l, \Z) \cong \Z^d$; let $S_N$ be the $(\Z/N)^d$ abelian cover of $S \setminus l$ with group $(\Z/N)^d$, and $S_N^{\mathrm{br}}$ (a compact $3$-manifold) the corresponding branched cover of $S$ at the link $l$.

Let $\Delta$ be the first non-zero Alexander polynomial of the link $l$, and $m(\Delta)$ its (additive) Mahler measure. Then,
$$
\lim_{N \to \infty} N^{-d} |H_1(S_N,\Z)|_{\mathrm{tors}} = \lim_{N \to \infty} N^{-d} |H_1(S_N^{\mathrm{br}},\Z)|_{\mathrm{tors}} = m(\Delta).
$$
The convergence is effective.
\end{corol}

The counterpart for the Betti numbers is Corollary~1.4 of Sarnak-Adams~\cite{sarnakadams}. Their paper also contains an example showing that the polynomials giving the Betti numbers may have a positive degree, unlike for the case of knots ($d = 1$).

\section{Complements, generalizations and related problems}  \label{complements}

\subsection{Subsets of complex embeddings}  \label{refo} With Liouville's general inequality, the trivial bound extends to any subset of the places: uniformly over all $\Sigma \subset \mathrm{Gal}(\Q(\mu_N)/\Q)$, it holds either
\begin{equation}  \label{liouville}
\sum_{\sigma \in \Sigma} -\log{|P(\boldsymbol{\zeta}^{\sigma})|} \ll_P \phi(N) \quad \textrm{ or } P(\boldsymbol{\zeta}) = 0.
\end{equation}
 Already with the polynomial $P(x) = x - 1$ this bound does not extend to an $o(\phi(N))$, because a positive proportion of the primitive $N$-th roots of unity belong to the $1/2$ neighborhood of $x = 1$. For $\Sigma$ constrained to satisfy $|\Sigma| = o(|G_N|)$ as $N \to \infty$, the hypothetical sub-Liouville $o(\phi(N))$ improvement over (\ref{liouville}) amounts to the question of whether  the Galois equidistribution of strict sequences of torsion points of $\mathbb{G}_m^d$ should extend to the singular test functions of the form $\log{|P|}$, where $P \in \Z[x_1^{\pm  1 },\ldots, x_d^{\pm 1}] \setminus \{0\}$. This would refine (B) and Silverman's general Conjecture~15 in~\cite{silverman}.

In this direction, the proof of Theorem~\ref{main} extends immediately to yield that for all $\Sigma \subset \mathrm{Gal}(\Q(\mu_N)/\Q)$, all but $O_{P,\varepsilon}(1)$ points $\boldsymbol{\zeta} \in \mu_N^d$ fulfil
\begin{equation}  \label{simultaneous}
\min_{\sigma \in \Sigma} -\log{|P(\boldsymbol{\zeta}^{\sigma})|} \leq \varepsilon \frac{\phi(N)}{|\Sigma|}.
\end{equation}

\subsection{No exceptional set}
In the case of $\Sigma = \{1\}$, the problem of no exceptional set reduces to the bound (\ref{strong}), and also to the statement that one may take $C(k,\varepsilon) = 0$ for $N \gg_{k,\varepsilon} 1$ in Corollary~\ref{corolsums}. This had been raised as an open ended question already in the 1980s, by Myerson~\cite{gerry}.  As in Feldman's refinement of Baker's theorem in the theory of logarithmic linear forms, Myerson asks (\ref{strong}) even in the stronger form of an $O_{P}(\log{N})$ bound. If true,  a simple packing argument shows this type of refinement of~(\ref{strong}) to be best possible; see, for example, Theorem~2 in Konyagin and Lev~\cite{konyaginlev}.  Habbeger's Theorem~\ref{hab} shows that such a bound does indeed hold for all but at most $\ll_{\epsilon} X^{\epsilon}$ of the primes $N < X$; see~\cite{habegger}.

\subsection{Myerson's conjecture}
  The refined convergence hypothesis (B) would  imply  Myerson's conjecture~\cite{gaussian}  about the asymptotic growth, as $q \to \infty$ for a fixed $t \mid q-1$, of the norm from $\Q(\mu_q)$ to $\Q$ of the `Gaussian period' $\eta := \sum_{u=0}^{t-1} \zeta_q^{us}$. For  $t \leq 3$ the asymptotic formula was proved by Myerson and refined by Duke~\cite{duke} to an apparently optimal error term.
 Myerson's conjecture is also discussed in chapter~10 of Konyagin and Shparlinski's book~\cite{konyaginchar}. Habegger~\cite{habnew} obtains a solution, with a refinement to an error term, of the case that $t$ is odd and $q$ is a prime.

\subsection{Small points} We have been concerned in this paper with Diophantine approximations by roots of unity.
 As is customary in the subject, one could ask if the hypothesis that the point $\boldsymbol{\zeta} \in \mathbb{G}_m^k(\bar{\Q})$ is torsion could be weakened to a hypothesis that $\boldsymbol{\zeta}$ is an algebraic point with sufficiently small canonical height. A straightforward modification of the proof, with Shou-wu Zhang's structural theorem on ``Bogomolov $\mathbb{G}_m$'' replacing the Laurent-Sarnak ``Lang $\mathbb{G}_m$'' theorem,  yields a unificaction of Theorem~\ref{main} and Zhang's theorem. We state here the non-uniform version, which follows immediately from the proof of Theorem~\ref{main} using Theorem~4.2.2 in Bombieri and Gubler's book~\cite{bg}, although it seems likely that $M$ and $\epsilon$ in the theorem may in fact be taken independent of the height of $P$, as soon as $[K:\Q]\ \gg_{d,\varepsilon,\deg{P},h(P)} 1$.

\begin{thm} \label{smallth}  Fix $P \in \Z[x_1,\ldots,x_d]$ and $\varepsilon > 0$, and assume that all components of $\{P = 0\}$ map onto under every surjective homomorphism $\mathbb{G}_m^d \twoheadrightarrow \mathbb{G}_m^{d-1}$. Then there are effectively computable functions
$$
M = M(d,\varepsilon,\deg{P},h(P)) < \infty, \quad \epsilon = \epsilon(d,\varepsilon,\deg{P},h(P)) > 0,
$$
and a finite union $Z \subsetneq \mathbb{G}_m^d$ of proper torsion cosets $\boldsymbol{\mu}T \subsetneq \mathbb{G}_m^d$,
such that the following is true:

For every finite extension $K/\Q$, and every subset $\Sigma \subset \mathrm{Aut}(K/\Q)$, all but at most $M$ points $\boldsymbol{\alpha} \in \mathbb{G}_m^d(K)$ of canonical height $ h(\boldsymbol{\alpha}) < \epsilon$ satisfy either
$$
  \min_{\sigma \in \Sigma} - \log{|P(\boldsymbol{\alpha}^{\sigma})|} \leq \varepsilon \frac{[K:\Q]}{|\Sigma|} \quad \textrm{or } \boldsymbol{\alpha} \in Z.
$$
\end{thm}

Here, $h : \mathbb{G}_m^k(\bar{\Q}) \to \R^{\geq 0}$ is Weil's absolute canonical logarithmic height. The uniformity feature $M = M(\varepsilon,\deg{P})$ and $\epsilon = \epsilon(\varepsilon,\deg{P})$, as soon as $[K:\Q] \gg_{P,\varepsilon} 1$, is immediate from the proof in the $d = 1$ case. Very possibly the same feature persists for all $d$, on using Theorem~4.2.3 in Bombieri and Gubler~\cite{bg} instead of Theorem~4.2.2 of {\it loc.cit.}, but that would require additional work.

\subsection{An example} \label{examp} The following example shows that, unlike possibly for the roots of unity case, an exceptional set is definitely needed in the general Theorem~\ref{smallth}.
  It is taken from the paper~\cite{bakerihrumely} of Baker, Ih and Rumely, with a similar consideration appearing in Autissier~\cite{autissier}.

  Consider the equation $x^n(x-2) = 1$. It is usually irreducible over $\Q$. For example, with $x$ replaced by $x+1$, it is Eisenstein at $2$ if $n = 2^k-1$, showing that it certainly  is irreducible for arbitrarily large $n$; fix a sequence of such $n$ going to infinity. By Rouch\'e's theorem, there is one root $\alpha_n$ close to $2$, and then the equation shows
$$
|2 - \alpha_n| = |\alpha_n|^{-n} < (2-o(1))^{-\deg{\alpha_n}}.
$$
On the other hand, the equation has bounded length and hence certainly $h(\alpha_n) \to 0$ as $n \to \infty$; cf. Bombieri-Gubler~\cite{bg}, Prop.~1.6.6 and Lemma~1.6.7. (Alternatively, note that the remaining conjugates of $\alpha_n$ are close to the unit circle, again by Rouch\'e's theorem.) Thus, already with $d = 1, P(x) = x-2, \varepsilon = 0.5$ and a single place, an $M > 0$ (non-empty exceptional set) is required for infinitely many $K$ in Theorem~\ref{smallth}.

 We note also that, by a theorem of Mignotte~\cite{mignotte}, the statement of Theorem~\ref{smallth} does hold with empty exceptional set ($M = 0$) for the case $P(x) = x-1$, as $[K:\Q] \gg_{\varepsilon} 1$. Sub-Liouville bounds of this type form a small subject of its own, which has come to be known as \emph{algebraic points close to one}; see \cite{mignotte,amoroso,amorososurvey,onelog,bug,dub}. It would be interesting to know whether a non-empty exceptional set is possible in Theorem~\ref{smallth} under the additional assumption (shared by the roots of unity) that the extensions $\mathbb{Q}(\boldsymbol{\alpha}) / \Q$ are Galois.

\subsection{$p$-adic metrics} Theorem~\ref{smallth} and its proof extends also to the $p$-adic places, and a similar example to the above, using the equation $x^n(x-1/p) = 1$, shows that even for $\Sigma = \{1\}$, the $p$-adic version too is best possible in this extension to small algebraic points. This is in stark contrast with the case of $p$-adic roots of unity (the analog of the $\Sigma = \{1\}$ case of Theorem~\ref{main}), where Tate and Voloch~\cite{tatevoloch} have shown that a torsion point not lying in a subvariety of $\mathbb{G}_{m/\C_p}^d$ is $p$-adically bounded away from the subvariety. For general $\Sigma \subset G_N$, Tate and Voloch's theorem shows that the $p$-adic variant of Theorem~\ref{smallth} does hold with a best possible bound and an empty exceptional set for the case of $p$-adic roots of unity.

\subsection{Algebraic dynamics, preperiodic points, and convergence to the Mahler measure} \label{prep} A generalization of Theorem~\ref{smallth} could be expected in the framework of algebraic dynamics. Let $f_1,\ldots,f_d : \proj^1 \to \proj^1$ be rational maps over $\bar{\Q}$ of degrees $q_i > 1$, and consider $\widehat{h}_f: \proj^1(\bar{\Q}) \to \R^{\geq 0}$ the dynamical height function $\widehat{h}_f(P) := \lim_{n \to \infty} q^{-n} h(f^n(P))$ of Call and Silverman; it is non-negative and vanishes exactly on the preperiodic points. One could ask about extending Theorem~\ref{smallth} to Diophantine approximation by points $\boldsymbol{\alpha}  \in (\proj^1)^d(K)$ having a small enough dynamical height $\widehat{h}_f(\boldsymbol{\alpha}) := \sum_{i=1}^d \widehat{h}_{f_i}(\alpha_i) < \epsilon$, in particular, to Diophantine approximation by the preperiodic points of $(f_1,\ldots,f_d)$.

In this direction, Szpiro and Tucker~\cite{szpirotucker} have used Roth's theorem to prove that the dynamical Mahler measure $\int_{\proj^1(\C)} \log{|P|} \, d\mu_f$ of an integer univariate polynomial $P \in \Z[x]$ is approached, as $n \to \infty$, by the averages of $\log{|P|}$ over either of the sets $\{\alpha \mid f^n(\alpha) = \alpha\}$ and $\{\alpha \mid f^n(\alpha) = \beta \}$, for any given $\beta \in \proj^1(\bar{\Q})$ that is not one of $\leq 2$ exceptional point for the rational iteration $f$. Here, $\mu_f$ is the Brolin-Lyubich measure: the unique invariant probability measure of maximum entropy for $(\proj^1(\C),f)$. It could be interesting to explore a higher dimensional case of this problem.
Note that, by the example in~\ref{examp},  convergence to the Mahler measure can fail for a sequence of small points, already for $f(x) = x^2$ and the Weil canonical height.

\subsection{Arakelov theory} One might also expect a more general conceptual framework for the Diophantine result of this paper. Given any non-negative height function $h_L : X(\bar{\Q}) \to \R^{\geq 0}$ arising by Arakelov's intersection theory from a non-trivial nef semipositive  adelically metrized line bundle on a projective variety $X/\Q$, it could be asked whether for every $\varepsilon > 0$ and a subvariety $E \subset X$ satisfying the appropriate non-degeneracy condition, there are $\epsilon = \epsilon(X,L,E,\varepsilon) > 0, M = M(X,L,E,\varepsilon) < \infty$, and a proper closed algebraic subset
$$
Z = Z(X,L,E,\varepsilon) \subsetneq X,
$$
 such that, for every number field $K$, all but at most $M$ points $x \in X(K) \setminus Z$ having $h_L(x) < \epsilon$ are of distance at least $e^{-\varepsilon \, [K:\Q]}$ from $E(\C)$.

\subsection{CM points} Instead of torsion points, one could consider these problems for different sets of special points in special varieties. In this direction, Habegger~\cite{habeggersin} has extended the Tate-Voloch theorem to the case of CM points in a power of the modular curve that are ordinary at a fixed finite prime $p$. He notes that the Tate-Voloch statement fails for supersingular CM points; however, one could expect the analog of (\ref{strong}) to hold also  for the supersingular points, as well as for the Archimedean valuations. Indeed, it is the $p$-ordinary points that are similar to $p$-adic roots of unity, by means of the Serre-Tate theory of canonical liftings which Habegger exploits in his paper. In contrast, the $p$-supersingular points are not discrete, and should behave similarly to the complex roots of unity.

It could be interesting to establish a version of Theorem~\ref{smallth} for singular moduli. Results in this direction are again implied by Habegger's recent work~\cite{habegger}, at least in the Archimedean situation.

\subsection{Diophantine approximation by closed orbits} The $\Sigma = \{1\}$ case of Theorems~\ref{main} and~\ref{smallth} are mostly trivial within a fixed Galois orbit (except for the uniformity clauses), and so these results may be regarded as bounding the number of Galois orbits of torsion or small points that get very near to a fixed subvariety. Likewise, \ref{prep} outlines  a similar hypothesis on how closely may a subvariety be approached by the periodic trajectories of a rational iteration on a power of the projective line. One could also ask for the metric forms of these results, in a wider dynamical context. For instance, what are the analogs of Dirichlet's and Khintchin's theorems for Diophantine approximation by preperiodic points in the Julia set of a rational function? For a generic point $P$ on a complete finite area surface of negative curvature, how closely is $P$ approached by a closed geodesic or a closed horocycle, in terms of the length of the geodesic or the horocycle?

\subsection{Schmidt's Subspace theorem} Finally, the strong parallel with the proofs of the theorems of Roth and Schmidt lead us to ask whether our Theorem~\ref{smallth} here could be merged into a common generalization with Schmidt's Subspace theorem or some of its noteworthy consequences. For example, in view of Laurent's theorem~\cite{laurent} on Diophantine approximation from a given finitely generated subgroup $\Gamma \subset \mathbb{G}_m^k(\bar{\Q})$, we could inquire about the existence of an $\epsilon = \epsilon(P,\Gamma,\varepsilon) > 0$ an $M(P,\Gamma,\varepsilon) < \infty$, and a finite union $Z = Z(P,\Gamma,\varepsilon) \subsetneq \mathbb{G}_m^d$ of proper torus cosets, such that, for all number fields $K$, the inequality
$$
|P(\boldsymbol{\alpha})| < \exp(- \varepsilon \, \big(1+h(\boldsymbol{\alpha}))[K:\Q] \big), \quad \boldsymbol{\alpha} \notin Z
$$
has at most $M(P,\Gamma,\varepsilon)$ $K$-rational solutions $\boldsymbol{\alpha} \in \mathcal{C}(\Gamma,\epsilon) \cap \mathbb{G}_m^d(K)$ from Habegger's ``truncated cone'' (from~\cite{boundedheight})
$$
\mathcal{C}(\Gamma,\epsilon) := \{ xy \in \mathbb{G}_m^k(\bar{\Q}) \mid x \in \Gamma, \, y \in \mathbb{G}_m^k(\bar{\Q}), \, h(y) \leq \epsilon(1+h(x)) \}.
$$
around $\Gamma$. This is in the spirit of Poonen's ``Mordell-Lang plus Bogomolov''~\cite{poonen}, now in the Diophantine approximation context of our paper.

\section{Logarithmic  forms}  \label{baker}

In this section we comment on the shortcomings of logarithmic linear forms theory towards verifying (B) over all product subgroups $\mu_{a_1} \times \cdots \times \mu_{a_d}$.
 The best general estimates for logarithmic linear forms are obtained from Baker's method, with refinements from the theory of zero multiplicity estimates on commutative group varieties. Baker and W\"ustholz's 1993 theorem~\cite{bakerwust} remains the state of the art in this subject, crowning a long sequence of step by step improvements over Baker's original 1966 result. An account of the Baker-W\"ustholz theorem is exposed in the short monograph~\cite{bakergold} by the same authors, which is now the authoritative reference for logarithmic linear forms theory. An improvement of the constants in the rational case, particularly regarding the exponential dependence in the number of logarithms, was further achieved by Matveev~\cite{matveevlf}.

One could naively hope that Theorems~\ref{hab} and~\ref{main} might combine with logarithmic linear forms theory to solve all $\mu_{a_1} \times \cdots \times \mu_{a_d}$ cases of (B), with the former covering the case that the group $\mu_{a_1} \times \cdots \times \mu_{a_d}$ is not almost cyclic and the latter covering the case that the group $\mu_{a_1} \times \cdots \times \mu_{a_d}$ is almost cyclic. This turns out to be not possible, in particular
the state of the art results in logarithmic linear forms theory are not nearly as strong as we need in the aspect of primary relevance to our problem, which is precisely the degree aspect of the number field generated by the arguments of the logarithmic forms.

 What concerns us from logarithmic linear forms theory is a sub-Liouville upper bound on a quantity of the form $-\log{|\alpha^b - 1|}$ (equivalently, on the rational linear form $\mathrm{dist}(b \log{\alpha}, 2\pi i \Z)$ in two logarithms $\log{\alpha}$ and $\log{1}$), for $b \in \N$ and for certain non-torsion algebraic points $\alpha \in \mathbb{G}_m(\bar{\Q})$ of degree $D$ having a bounded absolute (logarithmic) height $h(\alpha) \ll 1$. In this particular case, the Baker-W\"ustholz theorem amounts to a bound
\begin{equation} \label{genlog}
-\log{|\alpha^b - 1|} \ll D^4 \log{D} \cdot h'(\alpha) \log{b},
\end{equation}
provided the left-hand side is finite,
 where $h'(\alpha) := \max(h(\alpha),1/D)$, and the  implied constant is absolute (and small). The logarithmic dependence on $b$ is best possible as far as it goes, but the dependence on $D$ leaves a lot to be desired. Liouville's trivial bound is
 \begin{equation} \label{liou}
-\log{|\alpha^b - 1|} \leq Db \cdot h(\alpha) +  D\log{2} \ll_{1+ h(\alpha)} Db,
 \end{equation}
 and~(\ref{genlog}) is not a sub-Liouville bound unless $b > D^3$. While~(\ref{genlog}) does solve (B) in the case, for instance, of  the finite groups of the form $\mu_{a_1} \times \cdots \times \mu_{a_d}$ having  $a_d > (a_1 \cdots a_{d-1})^{7}$, the results of logarithmic forms theory do not appear capable of reaching Theorem~\ref{convergence} proved in this paper.

     The fact is that, in the preset stage of the subject, the $D^4$ power stands as a crucial barrier for any of Gelfond's, Schneider's or Baker's methods to succeed, with any $o(b)$ at all. As Matveev notes in~\cite{matveevlf} (in the introduction of the second paper), a similar $D^4$ dependence had arisen already Gelfond's work, as typified by Theorem~3.4.III of his 1952 book~\cite{gelfond}. A bound like~(\ref{genlog}) was first reached by Waldschmmidt~\cite{wald} in 1980, and since unimproved in the $D$ aspect. Mignotte and Waldschmidt~\cite{mignottewaldschmidt} have shown that Schneider's method leads to a $\ll D^4 h'(\alpha) (\log{b})^2$ version of~(\ref{genlog}), which Laurent~\cite{laurentint} then refined to a rather small numerical constant, using an interpolation determinant.

 One could conceivably\footnote{If one is optimistic enough to believe in Lehmer's conjecture.}  expect the bound~(\ref{genlog}) to refine to  an $O(D \cdot h(\alpha) \cdot \log{eb})$, for $\alpha \in \mathbb{G}_m(\bar{\Q})$ non-torsion. Clearly, such an optimistic statement would be best possible jointly in $D, h(\alpha)$ and~$b$. Moreover, in view of Dirichlet's Approximation theorem,  even the $b \gg_{\alpha} 1$ asymptotic form of this hypothetical best possible bound would force Salem numbers to be bounded away from $1$. This indicates that a refinement like this is hopelessly out of reach. It could nonetheless be worth remarking that a completely uniform statement of this strength would combine with Habegger's Theorem~\ref{hab} to solve (B) in all $\mu_{a_1} \times \cdots \times \mu_{a_d}$ cases, for arbitrary $d$.
Of course, even then many cases of finite subgroups would remain.

\section{Diophantine/Dynamical pairs}  \label{appo}

We conclude the paper with an overview of mathematically equivalent pairs of Diophantine and dynamical problems.

\subsection{}
   The brilliant example of this has been Margulis's solution of the longstanding Oppenheim conjecture in the arithmetic theory of real indefinite quadratic forms. This began with a recasting of the problem by Raghunathan in terms of orbit closures in homogeneous dynamics (an $\mathrm{SO}(2,1)$-orbit in $\mathrm{SL}(3,\R) / \mathrm{SL}(3,\Z)$ is either closed or dense), to which Margulis could apply a well developed ergodic arsenal of unipotent flows. Margulis's proof found its proper setting with Ratner's ensuing theorems on measure rigidity, algebraicity of orbit closures and equidistribution in the homogeneous dynamics of flows generated by ad-unipotent one-parameter subgroups of a Lie group (real or $p$-adic). These fundamental developments on the ergodic side have since found diverse applications back to number theory, including, for two notable instances, to Mazur's anticyclotomic  conjecture in elliptic curves Iwasawa theory (Vatsal~\cite{vatsal}), and the non-Poisson distribution of   gaps in the sequence $\sqrt{n} \mod{1}$ (Elkies-McMullen~\cite{gaps}).

 While Ratner's theorems  exclude the case of diagonalizable (torus) flows --- indeed, they patently fail for geodesic flows, --- such (\emph{hyperbolic}) flows and actions are of a still more basic Diophantine character, beginning with E. Artin's symbolic coding of modular geodesics via the continued fraction expansion. Furthermore, like in Furstenberg's $\times 2, \times 3$ problem, higher rank hyperbolic homogeneous dynamics is believed to regain a similar rigidity to the unipotent case. For the case of the Weyl chamber flow on $\mathrm{SL}(3,\R) / \mathrm{SL}(3,\Z)$, the orbit closure Margulis conjecture amounts precisely to another old problem of Diophantine approximations: Cassels and Swinnerton-Dyer's refinement~\cite{casselssd} of Littlewood's conjecture that $\inf_{n} n \| n\alpha \| \| n\beta \| = 0$ for all $\alpha, \beta \in \R$. This problem has remained unsolved from both its Diophantine and ergodic sides, even in the basic case $(\alpha,\beta) = (\sqrt{2},\sqrt{3})$. Cassels and Swinnerton-Dyer gave in their paper~\cite{casselssd} a Diophantine solution of the case that $\Q(\alpha,\beta)$ is a cubic extension of $\Q$, while much more recently Einsiedler, Katok and Lindenstrauss~\cite{ekl} developed an entropy method by which they could establish Littlewood's conjecture for all $(\alpha,\beta) \in \R^2$ outside of a set of Hausdorff dimension zero.

 \subsection{} With homogeneous flows, the applications so far have been from dynamics  to number theory: one finds a dynamical formulation of a Diophantine problem, which one seeks to approach  by ergodic theory methods.  (There are also a few cases where alternative proofs are available from either the ergodic or Diophantine sides.) The interaction, however, is genuinely organic, and does not limit to using ergodic theory as a tool for number theory. We leave aside here  the  ``intrinsic joins'' such as the Diophantine character of KAM theory or the currently developing subject of Arithmetic Dynamics, and focus purely  on pairs  of an {\it \`a priori} dynamical and an {\it \`a priori} Diophantine problems that turn out to be mathematically equivalent. Besides homogeneous dynamics, another setting for this is supplied by the robust class of higher rank systems studied in  Klaus Schmidt's book~\cite{schmidt}: the ones said to be of \emph{algebraic origin}, namely, the $\Z^d$-actions by automorphisms of a compact group. This is the context that our paper addresses.

  Here, the implications thus far discovered tend to go with Weyl's direction, from number theory to dynamics; and here too, the Diophantine questions were formulated and studied, for their intrinsic interest,  long before their  dynamical significance  could be found. The celebrated Lehmer problem, on the spectral gap in the Mahler measure (see Bombieri~\cite{bombieridist} for a perspective in Diophantine Geometry), was shown by Lind to be equivalent to either of the following dynamical questions: \cite{lind} does the infinite torus $(\R/\Z)^{\Z}$ admit an ergodic automorphism with finite entropy? \cite{lindskew} are all Bernoulli shifts {\it algebraizable}, in the form of a measurable equivalence with an automorphism of a compact group? Lind, Schmidt and Ward's paper~\cite{lindschmidtward} places this relation precisely in the general context of $\Z^d$-actions by automorphisms of a compact group, asking whether such dynamical systems are general or negligible from the measurable point of view. Lehmer's problem equates to an alternative for the set of entropies of such systems: either it is countable, or it is the full continuum $[0,\infty]$.

 \subsection{} These questions are, it seems, completely open ended.  A better understood  indication of the depth of the interaction, which regards strictly the higher rank case, is the equivalence of  W.M. Schmidt, A.J. Van der Poorten and H.P. Schlickewei's  theorem of the finiteness of non-degenerate solutions to the equation $x_1 + \cdots + x_r = 1$ in a finitely generated group $\Gamma \subset \C^{\times}$ (see Evertse and Gy\"ory's book~\cite{unit} for a comprehensive overview of this subject), and  K. Schmidt and T. Ward's theorem~\cite{schmidtward} about higher order mixing, see also chapter VIII of Schmidt's~\cite{schmidt} or chapter~8.2 of Einsiedler and Ward's books~\cite{einsidlerward}: a mixing $\Z^d$-action by automorphisms of a compact \emph{connected} abelian group  is mixing of all orders. The only known proof of this pair of equivalent results comes from number theory (via Weyl's criterion), with the Thue-Siegel-Roth-Schmidt method that we too  employ in this paper. A dynamical approach to higher order mixing, should it exist in a form amenable to a completely quantitative estimate, could be expected to shed a  light on a notorious Diophantine problem: to give an effective solution to the multivariate $S$-unit equation.

\subsection{} The Diophantine/dynamical pair (A) and (B) reflects the growth and distribution of periodic points of algebraic $\Z^d$-actions. Theorem~\ref{convergence} addresses this in the averaged sense of the cubical lattices $N \cdot \Z^d$. Refining this to arbitrary  $\Z^d \supset \Gamma \supset N \cdot \Z^d$ and, furthermore, to long orbits in $\mathrm{Per}_N(T)$, still eludes a proof, and Theorem~\ref{convergence} can be seen as an averaged form of growth rate and equidistribution.

Such ``averaged'' growth and equidistribution properties are, as is well known, typical of dynamical systems that exhibit at least a partial hyperbolicity. They hold for all weakly topologically mixing flows or transformations that satisfy Smale's Axiom~A. We refer to Margulis~\cite{margulis}, Bowen~\cite{bowen} and Parry and Pollicott~\cite{axioma,pol,parrypollicott} for a detailed study of the growth and equidistribution of the periodic trajectories of such flows. This subject was also inspired by number theory, but this time mainly on the level of analogy and methodology, the inspiration coming from Selberg's trace formula {\it vis \`a vis} the equidistribution of prime ideals in algebraic number fields $K$ over ray class fields or as points of the space of lattices in $K \otimes \R$.  Very recently, inspired by Sullivan's dictionary between Kleinian groups and the rational maps of complex dynamics, H. Oh and D. Winter~\cite{ohwinter} obtained a completely new case of hyperbolic equidistribution lying outside of the framework of Axiom A systems. Let us also mention that,  in arithmetically relevant cases, a refined equidistribution is sometimes available in which only the orbits of a fixed length are considered. Duke's theorem~\cite{duke} is a fine example of this,  powered by   a lower bound on the total length $h \log{\varepsilon}$ of the union of $h$ equal length closed geodesics corresponding to a given discriminant, that amounts to Siegel's ineffective theorem $L(1,\chi_q) \gg_{\epsilon} q^{-\epsilon}$ in prime number theory. In our application, equidistribution emerges analogously from a lower bound in terms of $N$ on the size of $\mathrm{Per}_N(T)$. The mechanism for this inference is in Ward's paper~\cite{wardexp} already quoted in the introduction, by means of commutative algebra and harmonic analysis on the group~$X$.

\subsection{} The results  of this paper appear to unify naturally into a common generalization with W. M. Schmidt's Subspace theorem. All these Diophantine results are semi-effective, regarding only the number of solutions to a Diophantine inequality. Section~\ref{complements} indicates some elements of such a unification, in addition to a few other extensions concerning points with sufficiently small canonical height. Our motivation in doing this is an  attempt to unravel the Diophantine nature of~(\ref{strong}), as well as the difficulties inherent in the quest for a higher dimensional  Baker theorem --- a problem whose central importance to number theory is stressed at the end of section~1.2 of Waldschmidt's book~\cite{waldschmidt}. In dynamics, we have seen  the two sides of the unification in, respectively, higher order mixing and the distribution of periodic points.

\subsection{} Besides the refinement to subgroups and long orbits in $\mathrm{Per}_N(T)$, an obvious  problem meriting a further study would be to extend the results on the distribution of periodic trajectories to algebraic actions or flows of other amenable groups, such as $\R^d$.
More  interesting but completely open ended is the question of reversing the direction of the application, by having a general dynamical paradigm brought to bear on the  unsolved problems in the Diophantine unification around A. Baker's and W. M. Schmidt's theorems.  This concerns higher dimensions in the former and effectivity in the latter, and would be especially rewarding with regard to mixing and effectivity. Halmos and Rokhlin's general question, of whether mixing yields an automatic higher order mixing for measure preserving transformations ($\Z$-actions), has turned out to reflect a surprisingly general phenomenon while  remaining one of the oldest unsolved problems of classical ergodic theory. The version  for the weak mixing property was obtained by Furstenberg over his ergodic proof of Szemer\'edi's theorem, see~\cite{furstenbergsz}, Th.~3.1. It could be also interesting to pursue the extension of Furstenberg's theorem to higher rank amenable group actions, and then the implications on the $S$-unit equation.

   An apparently fatal obstruction to a dynamical solution of the $S$-unit equation is the existence of counterexamples to automatic higher order strong mixing in the higher rank case. See Ledrappier~\cite{ledrappier} for the example that, according to Schmidt~\cite{schmidt}, has historically played a pivot role in the development of the whole subject of algebraic $\Z^d$-actions. But a closer examination of Ledrappier's system actually bolsters the case for viewing this classical number theory problem as embedded in a completely general context of mixing in ergodic theory.  Ledrappier's  example, $ \mathfrak{M} = R_2 / (2,1+x_1+x_2)$,  of a mixing algebraic $\Z^2$-action where a higher order mixing breaks down, has, along with its generalizations, been now understood to reflect a theorem in Diophantine Geometry over global function fields. Taking account of the effects of Frobenius in positive characteristic, similar structural theorems are available over function fields, leading in turn to corresponding structural theorems about mixing shapes: see Masser~\cite{masser}, Derksen and Masser~\cite{mixing}, as well as section VIII.28 of Schmidt~\cite{schmidt} and section~7.6 of Evertse and Gy\"ory~\cite{unit}. The combined Diophantine methods then answer completely the Halmos-Rokhlin problem for algebraic $\Z^d$-actions on any compact abelian group.

    As a matter of fact,  zero dimensional groups $X$ (the case in Ledrappier's system) are understood much more satisfactorily in this regard. Replacing Schmidt's Subspace theorem,  Julie Wang's effective truncated Second Main theorem~\cite{wang}, for projective space and a \emph{constant} divisor over a function field of any characteristic,  combines with Derksen and Masser's results to solve the higher order mixing problem effectively. Contrastingly, the Diophantine prototypes over  number fields have remained ineffective and much more mysterious. It remains to be seen whether a deepening of the dynamical connection would eventually be brought to bear, in this direction as well, on such outstanding questions of number theory.

\end{document}